\documentclass[12pt,draft]{article} 
\usepackage{amsmath,amssymb,amsthm,amsfonts, indentfirst}
\usepackage{enumerate,color,bm,graphicx,here}
\makeatletter
\def\@cite#1#2{[{{\bfseries #1}\if@tempswa , #2\fi}]}
\renewcommand{\section}{%
\@startsection{section}{1}{\z@}
{0.5truecm plus -1ex minus -.2ex}%
{1.0ex plus .2ex}{\bfseries\large}}
\def\@seccntformat#1{\csname the#1\endcsname.\ }
\makeatother

\numberwithin{equation}{section} 
\theoremstyle{theorem}
\newtheorem{thm}{Theorem}[section]

\newtheorem{lem}[thm]{Lemma}

\theoremstyle{definition}

\newtheorem{remark}{Remark}[section]

\newcommand{\ep}{\varepsilon}

\begin{document}
\footnote[0]
    {2020{\it Mathematics Subject Classification}\/. 
    Primary: 35K51; %Initial-boundary value problems for second-order parabolic systems
    Secondary: 35B44, %Blow-up in context of PDEs
                    92C17 %Cell movement (chemotaxis, etc.)
    }
     %chemotaxis; logistic source; finite-time blow-up; degenerate}
%==========================title==========================
\begin{center}
    \Large{{\bf Can chemotactic effects lead to blow-up or not in two-species chemotaxis-competition models?}}
\end{center}
\vspace{5pt}
%=========================author=========================
\begin{center}
    Masaaki Mizukami\footnote{Corresponding author}\\
    \vspace{10pt}
    Department of Mathematics, 
    Faculty of Education, Kyoto University of Education\\ 
    1, Fujinomori, Fukakusa, Fushimi-ku, Kyoto 612-8522, Japan\\
    \vspace{15pt}

    Yuya Tanaka\footnote{Partially supported by JSPS KAKENHI Grant Number JP22J11193.}\\
    \vspace{10pt}
    Department of Mathematics, 
    Tokyo University of Science\\
    1-3, Kagurazaka, Shinjuku-ku, 
    Tokyo 162-8601, Japan\\
    \vspace{15pt}

    Tomomi Yokota\footnote{Partially supported by JSPS KAKENHI Grant Number JP21K03278.}\\
    \vspace{10pt}
    Department of Mathematics, 
    Tokyo University of Science\\
    1-3, Kagurazaka, Shinjuku-ku, 
    Tokyo 162-8601, Japan\\
   \footnote[0]{
    E-mail: 
    {\tt masaaki.mizukami.math@gmail.com},
    {\tt yuya.tns.6308@gmail.com},
    {\tt yokota@rs.tus.ac.jp}
    }\\
    \vspace{2pt}
\end{center}
\begin{center}    
    \small \today
\end{center}

\vspace{2pt}
%=====================  Abstract  =======================
\newenvironment{summary}
{\vspace{.5\baselineskip}\begin{list}{}{%
     \setlength{\baselineskip}{0.85\baselineskip}
     \setlength{\topsep}{0pt}
     \setlength{\leftmargin}{12mm}
     \setlength{\rightmargin}{12mm}
     \setlength{\listparindent}{0mm}
     \setlength{\itemindent}{\listparindent}
     \setlength{\parsep}{0pt}
     \item\relax}}{\end{list}\vspace{.5\baselineskip}}
\begin{summary}
{\footnotesize {\bf Abstract.}
This paper deals with the two-species chemotaxis-competition models
  \begin{align*}
    \begin{cases}
      u_t = d_1 \Delta u 
                - \chi_1 \nabla \cdot (u \nabla w) 
                + \mu_1 u (1- u^{\kappa_1-1} - a_1 v^{\lambda_1-1}),
        &\quad x \in \Omega,\ t>0,\\
      v_t = d_2 \Delta v 
                - \chi_2 \nabla \cdot (v \nabla w) 
                + \mu_2 v (1- a_2 u^{\lambda_2-1} - v^{\kappa_2-1}),
        &\quad x \in \Omega,\ t>0,\\
      0 = d_3 \Delta w + \alpha u + \beta v - h(u,v,w),
        &\quad x \in \Omega,\ t>0, 
    \end{cases}
  \end{align*}
where $\Omega \subset \mathbb{R}^n$ $(n\ge2)$ is a bounded domain 
with smooth boundary, and 
$h=\gamma w$ or $h=\frac{1}{|\Omega|}\int_\Omega(\alpha u+ \beta v)\,dx$. 
In the case that $\kappa_1=\lambda_1=\kappa_2=\lambda_2=2$ 
and $h=\gamma w$, 
it is known that smallness conditions for the chemotacic effects 
lead to boundedness of solutions in e.g.\ \cite{M_2018_MMAS}. 
However, the case that the chemotactic effects are large 
seems not to have been studied yet; 
therefore it remains to consider 
the question whether the solution is bounded also 
in the case that the chemotactic effects are large. 
The purpose of this paper is to give a negative answer to this question.
}
\end{summary}
\vspace{10pt}

\newpage
%================================================================%
%================================================================%
%                                                Section 1                                                %
%                                              Introduction                                               %
%================================================================%
%================================================================%
\section{Introduction}

\noindent{\bf Background.} A Lotka--Volterra competition model represents 
a simplified scenario of the coexistence of two competing species. 
Moreover, if the species have random spacial motion, 
the model reads
  \[
    \begin{cases}
      u_t = d_1 \Delta u + \mu_1 u (1 - u - a_1 v), \\
      v_t = d_2 \Delta v + \mu_2 v (1 - a_2 u - v) 
    \end{cases}
  \]
(see e.g.\ \cite{Murray}), where 
$d_1, d_2, \mu_1, \mu_2, a_1, a_2 > 0$ are constants and 
$u,v$ idealize the population densities of the two species. 
In a mathematical view, the diffusion and competition terms 
lead to boundedness and stabilization of solutions. 
Indeed, there is a rich literature on large time behavior of solutions 
to the above model 
(see e.g.\ \cite{Brown, CS_1977, MR_1979, KY_1993, KW_1985, Lou-Ni, Matano-Mimura}). 
In addition to the random spacial motion,
if the model includes 
concentration phenomena of the two species, 
how do solutions behave?
In order to consider this question, we next focus on 
a two-species chemotaxis-competition model 
proposed by Tello and Winkler \cite{T-W_2012}, that is, 
  \[
    \begin{cases}
      u_t = d_1 \Delta u - \chi_1 \nabla \cdot (u \nabla w) + \mu_1 u (1- u - a_1 v),\\
      v_t = d_2 \Delta v - \chi_2 \nabla \cdot (v \nabla w) + \mu_2 v (1- a_2 u - v),\\
      \tau w_t = d_3 \Delta w + \alpha u + \beta v - \gamma w,
    \end{cases}
  \]
where $\chi_1, \chi_2, d_3, \alpha, \beta, \gamma>0$ and $\tau \in\{0,1\}$ are constants 
and $w$ shows the concentration of the signal substance. 
This model describes a situation 
in which multi species move toward higher concentrations of the signal substance, 
and moreover compete with each other. 
Here the structural feature of this model 
is that it involves not only the diffusion and competition terms 
but also the chemotaxis terms.
In particular, the diffusion and competition terms 
indicate stabilization of the species, whereas
the chemotaxis terms represent 
concentration of the species.
Therefore it is one of the fundamental themes 
whether the solutions remain bounded or not.

Now let us collect known results in the above models.
We shall recall the parabolic--parabolic--elliptic case ($\tau=0$). 
In this case, boundedness and stabilization were 
obtained under smallness conditions for $\chi_1$ and $\chi_2$ 
in the cases that $a_1, a_2 \in (0,1)$ (\cite{BLM, T-W_2012}) and 
that $a_1>1>a_2$ (\cite{STW}); 
moreover, these conditions 
in \cite{BLM, STW, T-W_2012} were improved 
in some cases in \cite{M_2018_MMAS}; 
particularly, the conditions for boundedness in \cite{M_2018_MMAS} were 
given by 
$0<\chi_1<\min\big\{ \frac{d_3 \mu_1}{\alpha},\frac{a_1 d_3 \mu_1}{\beta} \big\} 
                \cdot \frac{n}{n-2}$
and 
$0<\chi_2<\min\big\{ \frac{d_3 \mu_2}{\beta},\frac{a_2 d_3 \mu_2}{\alpha} \big\}
                \cdot \frac{n}{n-2}$.
Similarly, in the fully parabolic case ($\tau=1$), 
results on global existence and asymptotic stability were obtained 
under smallness conditions for $\chi_1$ and $\chi_2$
in \cite{B-W, LMW_2015, M_2017_DCDSB}.

\medskip

\noindent{\bf The challenge of proving blow-up.}
The above studies 
indicate that 
smallness of chemotactic effects entails 
boundedness and stabilization of solutions to 
chemotaxis-competition models.
However, to the best our knowledge, 
there is no result in the case that chemotactic effects are large. 
Therefore the following question arises:

  \begin{center}
    \mbox{\it Does the solution remain bounded 
              in the case that the chemotactic effects are large?}
  \end{center}

\noindent The purpose of this work is to give a {\it negative} answer to this question.

\medskip

\noindent{\bf Review of chemotaxis systems.} 
In order to consider the above question, 
it is useful to give an overview of known results for the  
parabolic--elliptic chemotaxis system with logistic source,
  \[
    \begin{cases}
      u_t = \Delta u - \chi \nabla \cdot (u \nabla w) + \lambda u-\mu u^\kappa,\\
      0 = \Delta w - w + u, 
    \end{cases}
  \]
where $\chi, \lambda, \mu>0$ and $\kappa >1$.
As to this model, results on boundedness and blow-up were obtained 
in \cite{B-F-L, F_2021_optimal, Tello-W-2007, W-2011, W-2018}. 
Indeed, Tello and Winkler \cite{Tello-W-2007} 
proved boundedness of solutions 
in the cases that $\kappa=2$ and $\chi<\frac{n}{(n-2)_+}\mu$ 
and that $\kappa>2$ and $\chi>0$,
whereas Winkler \cite{W-2018} derived finite-time blow-up 
under some smallness condition for $\kappa>1$; 
in a simplified parabolic--elliptic 
chemotaxis system with logistic source, 
smallness conditions for $\kappa$ leading to finite-time blow-up 
were established in higher-dimensional cases by Winkler \cite{W-2011} 
and in three- and four-dimensional cases by Black et al.\ \cite{B-F-L};
moreover, Fuest \cite{F_2021_optimal} succeeded 
in improving the conditions for $\kappa$ in \cite{B-F-L, W-2011} 
and showing finite-time blow-up under largeness conditions 
for $\chi$ in the case $\kappa=2$ in higher-dimensional cases.
In summary, in the one-species model,
boundedness and blow-up 
are determined by the sizes of $\chi$ and $\kappa$.

\medskip

\noindent{\bf Main results.} From the review of the 
one-species chemotaxis systems with logistic source, 
it is expected that a behavior of solutions depends on 
chemotactic effects and competitive effects
also in the two-species model.
In this context, we consider 
two-species chemotaxis models 
with generalized competition terms, 
which read 
  \begin{align}\label{P}
    \begin{cases}
      u_t = d_1 \Delta u 
                - \chi_1 \nabla \cdot (u \nabla w) 
                + \mu_1 u (1- u^{\kappa_1-1} - a_1 v^{\lambda_1-1}),
        &\quad x \in \Omega,\ t>0,\\[1mm]
      v_t = d_2 \Delta v 
                - \chi_2 \nabla \cdot (v \nabla w) 
                + \mu_2 v (1- a_2 u^{\lambda_2-1} - v^{\kappa_2-1}),
        &\quad x \in \Omega,\ t>0,\\[1mm]
      0 = d_3 \Delta w + \alpha u + \beta v - h(u,v,w),
        &\quad x \in \Omega,\ t>0,\\[2mm]
      \nabla u \cdot \nu = \nabla v \cdot \nu = \nabla w \cdot \nu = 0, 
        &\quad x \in \partial \Omega,\ t>0,\\[1mm]
      u(x,0) = u_0(x), \ 
      v(x,0) = v_0(x), 
        &\quad x \in \Omega,
    \end{cases}
  \end{align}
where $\Omega \subset \mathbb{R}^n$ $(n \ge 2)$ 
is a bounded domain with smooth boundary $\partial \Omega$; 
$\nu$ is the outward normal vector to $\partial \Omega$;
$d_1, d_2, d_3, \chi_1, \chi_2, \mu_1, \mu_2, a_1, a_2, \alpha, \beta>0$ and 
$\kappa_1, \kappa_2, \lambda_1, \lambda_2>1$; 
  \begin{align}\label{hdef}
    h(u,v,w) :=
      \begin{cases}
        \gamma w \quad(\gamma>0)
          &\text{(Keller--Segel type (cf.\ \cite{K-S_1971, K-S}))}, \\[2mm]
        \dfrac{1}{|\Omega|} \displaystyle\int_\Omega (\alpha u + \beta v)\,dx
          &\text{(J\"{a}ger--Luckhaus type (cf.\ \cite{J-L}))}; 
      \end{cases} 
  \end{align}
the initial data $u_0, v_0$ are assumed to be nonnegative functions; 
in particular, 
in the case that $h$ is of J\"{a}ger--Luckhaus type, 
we consider the model \eqref{P} with  
  \begin{align}\label{JLcondi}
    \int_\Omega w\,dx = 0, 
    \quad t > 0.
  \end{align}
From the preceding consideration, we aim to show 
boundedness and blow-up of solutions to \eqref{P}.
Here we briefly state our main results as follows:
  \begin{itemize}
    \item {\bf Prevention of blow-up, i.e., bundedness (Theorem \ref{thm; bdd}):} 
    If $n \ge 2$ and 
      \begin{align*}
        0<\chi_1<
          \begin{cases}
            \infty
              &\mbox{if}\ \kappa_1>2,\ \lambda_1>2,\\
            \frac{d_3 \mu_1}{\alpha} \cdot \frac{n}{n-2}
              &\mbox{if}\ \kappa_1=2,\ \lambda_1>2,\\
            \frac{a_1 d_3 \mu_1}{\beta} \cdot \frac{n}{n-2}
              &\mbox{if}\ \kappa_1>2,\ \lambda_1=2,\\
            \min \big\{ \frac{d_3 \mu_1}{\alpha}, \frac{a_1 d_3 \mu_1}{\beta} \big\}
            \cdot \frac{n}{n-2}
              &\mbox{if}\ \kappa_1=2,\ \lambda_1=2\\
          \end{cases}
      \end{align*}
    and
      \begin{align*}
        0<\chi_2<
          \begin{cases}
            \infty
              &\mbox{if}\ \kappa_2>2,\ \lambda_2>2,\\
            \frac{d_3 \mu_2}{\beta} \cdot \frac{n}{n-2}
              &\mbox{if}\ \kappa_2=2,\ \lambda_2>2,\\
            \frac{a_2 d_3 \mu_2}{\alpha} \cdot \frac{n}{n-2}
              &\mbox{if}\ \kappa_2>2,\ \lambda_2=2,\\
            \min \big\{ \frac{d_3 \mu_2}{\beta}, \frac{a_2 d_3 \mu_2}{\alpha} \big\}
              \cdot \frac{n}{n-2}
              &\mbox{if}\ \kappa_2=2,\ \lambda_2=2,
          \end{cases}
      \end{align*}
    then for all nonnegative initial data, 
    \eqref{P} has a bounded solution.
    \item {\bf Finite-time blow-up in the case 
             \boldmath{$h=\gamma w$} (Theorem \ref{thm; blow-up_KS}):} If
      \begin{align*}
        \max\{ \kappa_1, \lambda_1, \kappa_2, \lambda_2 \} < 
          \begin{cases}
            \frac{7}{6} 
              &\mbox{if}\ n \in \{ 3,4 \},\\
            1 + \frac{1}{2(n-1)} 
              &\mbox{if}\ n \ge 5,
          \end{cases}
      \end{align*}
    then for some nonnegative initial data, 
    \eqref{P} has a blow-up solution.
    \item {\bf Finite-time blow-up in the case 
             \boldmath{$h=\frac{1}{|\Omega|} \int_\Omega (\alpha u + \beta v)\,dx$}
             (Theorem \ref{thm; blow-up_JL}):} 
    If $n \ge 5$, $\lambda_1=\lambda_2=2$ and
      \begin{align*}
        \chi_1 > 
          \begin{cases}
            \frac{a_1 d_3 \mu_1}{\beta} \cdot \frac{n}{n-4}
              &\mbox{if}\ \kappa_1<2,\\
            \max \big\{ \frac{d_3 \mu_1}{\alpha}, \frac{a_1 d_3 \mu_1}{\beta} \big\}
              \cdot \frac{n}{n-4}
              &\mbox{if}\ \kappa_1=2
          \end{cases}
          \quad \mbox{and} \quad 
        \chi_2 > \frac{a_2 d_3 \mu_2}{\alpha}
      \end{align*}
    or \quad $n \ge 5$, $\lambda_1=\lambda_2=2$ and
      \begin{align*}
       \chi_1 > \frac{a_1 d_3 \mu_1}{\beta}
          \quad \mbox{and} \quad 
        \chi_2> 
          \begin{cases}
            \frac{a_2 d_3 \mu_2}{\alpha} \cdot \frac{n}{n-4}
              &\mbox{if}\ \kappa_2<2,\\
            \max \big\{ \frac{d_3 \mu_2}{\beta}, \frac{a_2 d_3 \mu_2}{\alpha} \big\}
              \cdot \frac{n}{n-4}
              &\mbox{if}\ \kappa_2=2, 
          \end{cases} 
      \end{align*}
    then for some nonnegative initial data, 
    \eqref{P} has a blow-up solution.
  \end{itemize}

\noindent{\bf Main ideas and plan of the paper.} We first collect 
some preliminary facts on 
local existence of solutions to \eqref{P} 
and estimates for masses of $u$ and $v$ in Section \ref{preli}. 
In Section \ref{secbdd} 
we show boundedness of solutions to \eqref{P}, 
the proof of which is based on \cite{M_2018_MMAS}. 
The key point is to establish an $L^p$-estimate for  
$u$ and an $L^q$-estimate for $v$ 
by constructing the differential inequalities 
  \[
    \frac{1}{p} \cdot \frac{d}{dt} \int_\Omega u^p\,dx 
    \le C_1 \int_\Omega u^p\,dx 
         - C_2 \int_\Omega u^{p+\kappa_1-1}\,dx
         + C_3
  \]
and 
  \[
    \frac{1}{q} \cdot \frac{d}{dt} \int_\Omega v^q\,dx 
    \le C_4 \int_\Omega v^q\,dx 
         - C_5 \int_\Omega v^{q+\kappa_2-1}\,dx
         + C_6.
  \]
In Section \ref{sectionKS} 
we prove finite-time blow-up 
in the case that $h$ is of Keller--Segel type. 
Referring to \cite{W-2018}, we will introduce the functions 
  \[
    \phi_U(t) := \int^{s_0}_{0} s^{-b}(s_0-s) U(s,t)\,ds
  \quad\mbox{and}\quad
    \phi_V(t) := \int^{s_0}_{0} s^{-b}(s_0-s) V(s,t)\,ds
  \]
with $s_0 \in (0,R^n)$ and $b \in (0,1)$, where 
  \[
    U(s,t) := \int^{s^\frac{1}{n}}_0 \rho^{n-1} u(\rho,t)\,d\rho
  \quad\mbox{and}\quad
    V(s,t) := \int^{s^\frac{1}{n}}_0 \rho^{n-1} v(\rho,t)\,d\rho.
  \]
In this case the goal is to obtain 
a superlinear differential inequality for $\phi_U+\phi_V$, 
and toward this goal, we need to control the integrals 
  \[
    \int^{s_0}_0 s^{-b}(s_0-s) \left( \int^s_0 U_s(\sigma,t)
    V_s^{\lambda_1-1}(\sigma,t)\,d\sigma \right)\,ds
  \]
and
  \[
    \int^{s_0}_0 s^{-b}(s_0-s) \left( \int^s_0 V_s(\sigma,t)
    U_s^{\lambda_2-1}(\sigma,t)\,d\sigma \right)\,ds
  \]
which come from the competition terms. 
To this end, we derive pointwise upper estimates for $u$ and $v$ 
(see Lemma \ref{lemprofile}) by an argument similar to that in 
\cite[Lemma 3.3]{W-2018}.
Finally, in Section \ref{sectionJL} 
we show finite-time blow-up 
in the case that $h$ is of J\"{a}ger--Luckhaus type. 
In this case, by the structure of the third equation in \eqref{P}, 
it is sufficient to establish 
a superlinear differential inequality for $\phi_U$ or $\phi_V$  
with the parameter $b>1$.
The key to obtaining the inequality for $\phi_U$ or $\phi_V$ is 
to derive concavity of $U$ and $V$ (see Lemma \ref{increase}).  
Here, this property will be shown by using a comparison principle, 
which is a different method in \cite[Lemma 2.2]{W-2018_nonlinear}.

%================================================================%
%================================================================%
%                                                Section 2                                                %
%================================================================%
%================================================================%
\section{Preliminaries}\label{preli}

We first recall the known result about local existence of solutions to \eqref{P}, 
which is proved by a standard fixed point argument 
as in \cite{C-W, STW}.

\begin{lem}\label{localsol}
Let $\Omega \subset \mathbb{R}^n$ $(n \ge 2)$ 
be a bounded domain with smooth boundary and 
$d_1, d_2, d_3, \chi_1, \chi_2, \mu_1, \mu_2, a_1, a_2, \alpha, \beta>0$, 
$\kappa_1, \kappa_2, \lambda_1, \lambda_2>1$ 
and let $h$ satisfy \eqref{hdef}. 
Assume that 
  \begin{align}\label{initial}
    u_0, v_0 \in C^0(\overline{\Omega}) \ \mbox{are nonnegative.}
  \end{align}
Then there exist $T_{\rm max} \in (0,\infty]$ and a unique triplet
$(u,v,w)$ of functions 
  \[
    u,v,w \in C^0(\overline{\Omega} \times [0,T_{\rm max})) 
                 \cap C^{2,1}(\overline{\Omega} \times (0,T_{\rm max})),
  \]
which solves \eqref{P} classically, 
and 
fulfills \eqref{JLcondi} in the case that $h$ is of J\"{a}ger--Luckhaus type.
Moreover, $u,v\ge0$ in $\Omega \times (0,T_{\rm max})$ and 
  \begin{align}\label{criterion}
    \mbox{if}\ T_{\rm max}<\infty, 
    \quad \mbox{then} \quad
    \lim_{t \nearrow T_{\rm max}} 
      ( \| u(\cdot,t) \|_{L^\infty(\Omega)} + \| v(\cdot,t) \|_{L^\infty(\Omega)} )
    = \infty.
  \end{align}
Furthermore, if $u_0, v_0$ are radially symmetric, then so are $u,v,w$ 
for any $t \in (0,T_{\rm max})$.
\end{lem}

We next show that the masses of $u$ and $v$ 
can be controlled by $\mu_1, u_0$ and $\mu_2, v_0$, respectively.

\begin{lem}\label{L1esti}
Assume that $u_0, v_0 \in C^0(\overline{\Omega})$ are nonnegative. 
Then
  \[
    \int_\Omega u(x,t)\,dx \le e^{\mu_1 t} \int_\Omega u_0(x)\,dx
    \quad \mbox{and} \quad
    \int_\Omega v(x,t)\,dx \le e^{\mu_2 t} \int_\Omega v_0(x)\,dx
  \]
for all $t \in (0,T_{\rm max})$.
\end{lem}
%-----------------------------------------------------------------%
\begin{proof}
By integrating the first and second equations in \eqref{P},  
and using that $\mu_1>0$ and $\mu_2>0$, it follows that 
  \[
    \frac{d}{dt}\int_\Omega u\,dx \le \mu_1 \int_\Omega u\,dx
    \quad \mbox{and} \quad
    \frac{d}{dt}\int_\Omega v\,dx \le \mu_2 \int_\Omega v\,dx
  \]
for all $t\in(0,T_{\rm max})$.
These inequalities imply the conclusion of this lemma.
\end{proof}

%================================================================%
%================================================================%
%                                                Section 3                                                %
%================================================================%
%================================================================%
\section{Boundedness}\label{secbdd}

In this section let us denote by $(u,v,w)$ 
the local classical solution of \eqref{P} on $[0,T_{\rm max})$
given in Lemma \ref{localsol}. 
We shall prove the following theorem 
which asserts global existence and boundedness in \eqref{P}.

%================================================================%
%                                              Theorem 3.1                                              %
%================================================================%
\begin{thm}\label{thm; bdd}
Let $\Omega \subset \mathbb{R}^n$ $(n \ge 2)$ 
be a bounded domain with smooth boundary, and let
$d_1, d_2, d_3, \chi_1, \chi_2, a_1, a_2, \alpha, \beta>0$. 
Assume that $\mu_1$ and $\mu_2$ fulfill
  \begin{align}\label{bddcondi1}
        0<\chi_1<
          \begin{cases}
            \infty
              &\mbox{if}\ \kappa_1>2,\ \lambda_1>2,\\
            \frac{d_3 \mu_1}{\alpha} \cdot \frac{n}{n-2}
              &\mbox{if}\ \kappa_1=2,\ \lambda_1>2,\\
            \frac{a_1 d_3 \mu_1}{\beta} \cdot \frac{n}{n-2}
              &\mbox{if}\ \kappa_1>2,\ \lambda_1=2,\\
            \min \big\{ \frac{d_3 \mu_1}{\alpha}, \frac{a_1 d_3 \mu_1}{\beta} \big\}
            \cdot \frac{n}{n-2}
              &\mbox{if}\ \kappa_1=2,\ \lambda_1=2
          \end{cases}
  \end{align}
and
  \begin{align}\label{bddcondi2}
        0<\chi_2<
          \begin{cases}
            \infty
              &\mbox{if}\ \kappa_2>2,\ \lambda_2>2,\\
            \frac{d_3 \mu_2}{\beta} \cdot \frac{n}{n-2}
              &\mbox{if}\ \kappa_2=2,\ \lambda_2>2,\\
            \frac{a_2 d_3 \mu_2}{\alpha} \cdot \frac{n}{n-2}
              &\mbox{if}\ \kappa_2>2,\ \lambda_2=2,\\
            \min \big\{ \frac{d_3 \mu_2}{\beta}, \frac{a_2 d_3 \mu_2}{\alpha} \big\}
              \cdot \frac{n}{n-2}
              &\mbox{if}\ \kappa_2=2,\ \lambda_2=2
          \end{cases}
  \end{align}
and that \eqref{initial} is satisfied. 
Then the corresponding solution $(u,v,w)$ 
of \eqref{P} 
is global and bounded in the sense that $T_{\rm max}=\infty$ and 
  \[
    \| u(\cdot,t) \|_{L^\infty(\Omega)} 
    + \| v(\cdot,t) \|_{L^\infty(\Omega)} 
    \le C 
  \]
for all $t \ge 0$ with some $C>0$. 
\end{thm}

\begin{remark} 
  In the case that $\kappa_1=\lambda_1=\kappa_2=\lambda_2=2$,  
  the conditions \eqref{bddcondi1} and \eqref{bddcondi2} are 
  the same as in \cite[(6)]{M_2018_MMAS}. 
  Thus this theorem is a generalization of \cite[Theorem 1.1]{M_2018_MMAS}.
\end{remark}

For the proof of Theorem \ref{thm; bdd} 
we derive an $L^p$-estimate for $u$ with some $p>\frac{n}{2}$ and 
an $L^q$-estimate for $v$ with some $q>\frac{n}{2}$.

%================================================================%
%                                               Lemma 3.2                                                %
%================================================================%
\begin{lem}\label{lem; Lpbdd}
Assume that \eqref{initial} holds. 
\begin{enumerate}[{\rm (i)}]
\item 
If $\mu_1$ satisfies \eqref{bddcondi1}, 
then for some $p>\frac{n}{2}$
there exists $C_1>0$ such that
  \begin{align}\label{Lpbdd of u}
    \| u(\cdot,t) \|_{L^p(\Omega)} \le C_1
  \end{align}
for all $t \in (0,T_{\rm max})$. 

\item
If $\mu_2$ satisfies \eqref{bddcondi2}, 
then for some $q>\frac{n}{2}$
there exists $C_2>0$ such that 
  \begin{align}\label{Lpbdd of v}
    \| v(\cdot,t) \|_{L^q(\Omega)}  \le C_2
  \end{align}
for all $t \in (0,T_{\rm max})$.
\end{enumerate}
\end{lem}
%------------------------------proof------------------------------%
\begin{proof}
We let $p>\frac{n}{2}$ be fixed later. 
Multiplying the first equation in \eqref{P} by $u^{p-1}$ 
and integrating it over $\Omega$, we have
  \begin{align*}
    \frac{1}{p} \cdot \frac{d}{dt} \int_\Omega u^p\,dx 
    &= d_1 \int_\Omega u^{p-1} \Delta u\,dx
         - \chi_1 \int_\Omega u^{p-1} \nabla \cdot (u \nabla w)\,dx 
    \\ 
    &\quad\,
         + \mu_1 \int_\Omega u^p(1-u^{\kappa_1-1}-a_1v^{\lambda_1-1})\,dx
    \\ 
    &= - \frac{4(p-1)}{p^2} d_1 \int_\Omega | \nabla u^\frac{p}{2} |^2\,dx 
         - \frac{p-1}{p} \chi_1 \int_\Omega u^p \Delta w\,dx 
    \\ 
    &\quad\,
         + \mu_1 \int_\Omega u^p\,dx 
         - \mu_1 \int_\Omega u^{p+\kappa_1-1}\,dx 
         - a_1 \mu_1 \int_\Omega u^p v^{\lambda_1-1}\,dx 
  \end{align*}
for all $t \in (0,T_{\rm max})$.
From the third equation in \eqref{P} 
and the nonnegativity of $h$, it follows that 
  \[
    - \frac{p-1}{p} \chi_1 \int_\Omega u^p \Delta w\,dx
    \le 
       \frac{p-1}{p} \cdot \frac{\alpha \chi_1}{d_3}  \int_\Omega u^{p+1}\,dx 
    + \frac{p-1}{p} \cdot \frac{\beta \chi_1}{d_3} \int_\Omega u^p v\,dx 
  \]
and so, we see that 
  \begin{align}\label{up}
    \frac{1}{p} \cdot \frac{d}{dt} \int_\Omega u^p\,dx 
    &\le - \frac{4(p-1)}{p^2} d_1 \int_\Omega | \nabla u^\frac{p}{2} |^2\,dx 
    \notag\\ 
    &\quad\,
         + \frac{p-1}{p} \cdot \frac{\alpha \chi_1}{d_3}  \int_\Omega u^{p+1}\,dx 
         + \frac{p-1}{p} \cdot \frac{\beta \chi_1}{d_3} \int_\Omega u^p v\,dx 
    \notag\\ 
    &\quad\,
         + \mu_1 \int_\Omega u^p\,dx 
         - \mu_1 \int_\Omega u^{p+\kappa_1-1}\,dx 
         - a_1 \mu_1 \int_\Omega u^p v^{\lambda_1-1}\,dx 
  \end{align}
for all $t \in (0,T_{\rm max})$. 
We first prove \eqref{Lpbdd of u} in the case that $\kappa_1=2$ and $\lambda_1=2$. 
In this case, the above inequality gives 
  \begin{align*}
    \frac{1}{p} \cdot \frac{d}{dt} \int_\Omega u^p\,dx 
    &\le \mu_1 \int_\Omega u^p\,dx 
         - \left( 
             \mu_1 - \frac{p-1}{p} \cdot \frac{\alpha \chi_1}{d_3} 
            \right)
            \int_\Omega u^{p+1}\,dx 
    \\
    &\quad\,
         - \left(
              a_1 \mu_1 -  \frac{p-1}{p} \cdot \frac{\beta \chi_1}{d_3} 
            \right) 
            \int_\Omega u^p v\,dx. 
  \end{align*}
From \eqref{bddcondi1} we can choose $p>\frac{n}{2}$ fulfilling 
  \begin{align}\label{kappa_1=2}
    \mu_1 - \frac{p-1}{p} \cdot \frac{\alpha \chi_1}{d_3} > 0
  \end{align}
and
  \begin{align}\label{lambda_1=2}
    a_1 \mu_1 -  \frac{p-1}{p} \cdot \frac{\beta \chi_1}{d_3} > 0.
  \end{align}
Then, noting that 
$- \big(
              a_1 \mu_1 -  \frac{p-1}{p} \cdot \frac{\beta \chi_1}{d_3} 
            \big) 
            \int_\Omega u^p v\,dx\le0$,
by using H\"{o}lder's inequality we derive that 
  \begin{align*}
    \frac{1}{p} \cdot \frac{d}{dt} \int_\Omega u^p\,dx 
    \le \mu_1 \int_\Omega u^p\,dx 
         - c_1 \left( \int_\Omega u^p\,dx \right)^{p+1} 
  \end{align*}
for all $t \in (0,T_{\rm max})$, where 
$c_1:=\big( \mu_1 - \frac{p-1}{p} \cdot \frac{\alpha \chi_1}{d_3} \big)|\Omega|^{-1} > 0$. 
Hence \eqref{Lpbdd of u} results from this.
Next, we consider the case that $\kappa_1 > 2$ and $\lambda_1 > 2$. 
For all $p > \frac{n}{2}$ it follows from Young's inequality that 
  \begin{align}\label{young1}
    \frac{p-1}{p} \cdot \frac{\alpha \chi_1}{d_3}  \int_\Omega u^{p+1}\,dx
    \le \frac{\mu_1}{2} \int_\Omega u^{p+\kappa_1-1}\,dx + c_2
  \end{align}
with some $c_2>0$, and 
  \begin{align}\label{young2}
    \frac{p-1}{p} \cdot \frac{\beta \chi_1}{d_3} \int_\Omega u^p v\,dx 
    \le a_1 \mu_1 \int_\Omega u^p v^{\lambda_1-1}\,dx 
         + c_3 \int_\Omega u^p\,dx
  \end{align}
with some $c_3>0$. 
Combining these inequalities with \eqref{up}, we obtain 
  \[
    \frac{1}{p} \cdot \frac{d}{dt} \int_\Omega u^p\,dx 
    \le c_4 \int_\Omega u^p\,dx 
         - \frac{\mu_1}{2} \int_\Omega u^{p+\kappa_1-1}\,dx
         + c_2
  \]
for all $t \in (0,T_{\rm max})$, where $c_4:=\mu_1+c_3$, 
and this leads to \eqref{Lpbdd of u}. 
Similarly, taking $p$ satisfying \eqref{kappa_1=2} and using \eqref{young1} 
in the case that $\kappa_1=2$ and $\lambda_1>2$, 
and picking $p$ fulfilling \eqref{lambda_1=2} and applying \eqref{young2} 
in the case that $\kappa_1>2$ and $\lambda_1=2$, 
we can confirm that \eqref{Lpbdd of u} holds. 
Moreover, we can show 
that \eqref{Lpbdd of v} holds for all $t \in (0,T_{\rm max})$ with some $q>\frac{n}{2}$ 
by an argument similar to that in the proof of \eqref{Lpbdd of u} 
under the condition \eqref{bddcondi2}.
\end{proof}

With these estimates for $u$ and $v$ at hand, 
we can show Theorem \ref{thm; bdd}.

%================================================================%
%                                        Proof of Theorem 3.1                                        %
%================================================================%
\begin{proof}[Proof of Theorem \ref{thm; bdd}]
Since the proof is similar to that in 
\cite[Lemma 2.4]{M_2018_MMAS}, 
let us briefly state the proof.
By making use of Lemma \ref{lem; Lpbdd}, 
we can establish 
an $L^p$-estimate for $u$ for some $p>\frac{n}{2}$ 
and an $L^q$-estimate for $v$ for some $q>\frac{n}{2}$. 
Thus,  
for all 
$r \in \big(\frac{n}{2}, \min\{p,q\}\big] \cap (0,n)$
there exists $c_1>0$ such that
$\| u(\cdot,t) \|_{L^r(\Omega)} + \| v(\cdot,t) \|_{L^r(\Omega)} \le c_1$
for all $t \in (0,T_{\rm max})$. 
Hence, applying elliptic regularity theory 
(see e.g.\ \cite[Theorem 19.1]{Friedman})
to the third equation in \eqref{P} 
and the Sobolev embedding theorem, we can take $c_2>0$ such that 
  \[
    \| \nabla w(\cdot,t) \|_{L^{\frac{nr}{n-r}}(\Omega)} \le c_2
  \]
for all $t \in (0,T_{\rm max})$,
and 
moreover,
we invoke 
the standard semigroup technique 
(see e.g.\ \cite[Proof of Lemma 3.2]{B-B-T-W})
to find $c_3>0$ such that 
  \[
    \| u(\cdot,t) \|_{L^\infty(\Omega)} 
    + \| v(\cdot,t) \|_{L^\infty(\Omega)} 
    \le c_3
  \]
for all $t \in (0,T_{\rm max})$, which in conjunction with \eqref{criterion} implies the end of the proof.
\end{proof}

%================================================================%
%================================================================%
%                                                Section 4                                                %
%================================================================%
%================================================================%
\section{Finite-time blow-up in a model of Keller--Segel type}\label{sectionKS}
In this section we consider the system \eqref{P} 
in the case that $h$ is of Keller--Segel type 
as follows:
  \begin{align}\label{P1}
    \begin{cases}
      u_t = d_1 \Delta u 
                - \chi_1 \nabla \cdot (u \nabla w) 
                + \mu_1 u (1- u^{\kappa_1-1} - a_1 v^{\lambda_1-1}),
        &\quad x \in \Omega,\ t>0,\\
      v_t = d_2 \Delta v 
                - \chi_2 \nabla \cdot (v \nabla w) 
                + \mu_2 v (1- a_2 u^{\lambda_2-1} - v^{\kappa_2-1}),
        &\quad x \in \Omega,\ t>0,\\
      0 = d_3 \Delta w + \alpha u + \beta v - \gamma w,
        &\quad x \in \Omega,\ t>0,\\
      \nabla u \cdot \nu = \nabla v \cdot \nu = \nabla w \cdot \nu = 0, 
        &\quad x \in \partial \Omega,\ t>0,\\
      u(x,0) = u_0(x), \ 
      v(x,0) = v_0(x), 
        &\quad x \in \Omega,
    \end{cases}
  \end{align}
where $\Omega = B_R(0) \subset \mathbb{R}^n$ $(n\ge3)$ 
is a ball with some $R>0$, and 
$u_0, v_0$ satisfy \eqref{initial} and
  \begin{align}\label{initial1}
    u_0, v_0 \ \mbox{are radially symmetric}. 
  \end{align}
Moreover, let $(u,v,w)$ be the local classical solution of \eqref{P1} 
on $[0,T_{\rm max})$
given in Lemma \ref{localsol}. 
We now state the following theorem which guarantees finite-time blow-up.

%================================================================%
%                                              Theorem 4.1                                              %
%================================================================%
\begin{thm}\label{thm; blow-up_KS}
Let $\Omega = B_R(0) \subset \mathbb{R}^n$ $(n\ge3)$ 
and
let $d_1, d_2, d_3, \chi_1, \chi_2, \mu_1, \mu_2, 
a_1, a_2>0$ and $\alpha, \beta, \gamma>0$ 
as well as $\kappa_1,\kappa_2,\lambda_1,\lambda_2>1$.
Assume that $\kappa_1, \kappa_2, \lambda_1$ and $\lambda_2$ satisfy that 
      \begin{align}\label{KSblowup}
        \max\{ \kappa_1, \lambda_1, \kappa_2, \lambda_2 \}
        < 
        \begin{cases}
          \frac{7}{6} &\mbox{if}\ n \in \{3,4\},\\
          1+\frac{1}{2(n-1)} &\mbox{if}\ n \ge5.
        \end{cases}
      \end{align}
    Then, for all $L>0$, $M_0>0$ and $\widetilde{M}_0 \in (0,M_0)$ 
    there exists $r_\star \in (0,R)$ with the following property\/{\rm :}
    If $u_0, v_0$ satisfy \eqref{initial}, \eqref{initial1} and
      \begin{align}\label{initial KS} 
        \int_\Omega ( u_0(x) + v_0(x) )\,dx = M_0 
        \quad \mbox{and} \quad 
        \int_{B_{r_\star}(0)} ( u_0(x) + v_0(x) )\,dx \ge \widetilde{M}_0
      \end{align}
    as well as
      \begin{align}\label{initial_profile}
        u_0(x) + v_0(x) \le L |x|^{-n(n-1)} \quad\mbox{for all}\ x \in \Omega,
      \end{align}
    then the corresponding solution $(u,v,w)$ of \eqref{P1}  
    blows up in finite time in the sense that $T_{\rm max}<\infty$ and 
      \begin{align}\label{blowupKS}
        \lim_{t \nearrow T_{\rm max}} 
        ( \| u(\cdot,t) \|_{L^\infty(\Omega)} + \| v(\cdot,t) \|_{L^\infty(\Omega)} )
        = \infty. 
      \end{align}
\end{thm}

\begin{remark}
  If $v=0$ in \eqref{P1}, then $\lambda_1, \kappa_2$ and $\lambda_2$ can be
  taken smaller than $\kappa_1$, and thus \eqref{KSblowup} is rewritten as
  \[
    \kappa_1< 
      \begin{cases}
        \frac{7}{6} &\mbox{if}\ n \in \{3,4\},\\
        1+\frac{1}{2(n-1)} &\mbox{if}\ n \ge5.
      \end{cases}
  \]
  This is the same condition as in \cite[(1.4)]{W-2018}, 
  which means that the result in \cite[Theorem 1.1]{W-2018} 
  is generalized to that in the two-species model.
\end{remark}

The proof of Theorem \ref{thm; blow-up_KS} is based on \cite{W-2018}. 
In what follows,
we regard $u(x,t), v(x,t)$ and $w(x,t)$ 
as functions of $r:=|x|$ and $t$.
Also, we introduce the mass accumulation functions $U, V$ and $W$ as 
  \begin{align*}
    U(s,t) &:= \int^{s^\frac{1}{n}}_0 \rho^{n-1} u(\rho,t)\,d\rho,
    \\ 
    V(s,t) &:= \int^{s^\frac{1}{n}}_0 \rho^{n-1} v(\rho,t)\,d\rho
  \end{align*}
and
  \[
    W(s,t) := \int^{s^\frac{1}{n}}_0 \rho^{n-1} w(\rho,t)\,d\rho
  \]
for $s \in [0,R^n]$ and $t \in [0,T_{\rm max})$. 
Moreover, we define the functions $\phi_U$, $\phi_V$, $\psi_U$ and $\psi_V$ as 
  \begin{align*}
    \phi_U(t) &:= \int^{s_0}_{0} s^{-b}(s_0-s) U(s,t)\,ds,\\
    \phi_V(t) &:= \int^{s_0}_{0} s^{-b}(s_0-s) V(s,t)\,ds
  \end{align*}
and
  \begin{align*}
    \psi_U(t) &:= \int^{s_0}_0 s^{-b}(s_0-s) U(s,t) U_s(s,t)\,ds,\\
    \psi_V(t) &:= \int^{s_0}_0 s^{-b}(s_0-s) V(s,t) V_s(s,t)\,ds
  \end{align*}
for $t \in [0,T_{\rm max})$ 
with suitably chosen $s_0 \in (0,R^n)$ and $b \in (0,1)$. 
Here, we note that $\phi_U, \phi_V \in C^0([0,T_{\rm max})) \cap C^1((0,T_{\rm max}))$.

To prove Theorem \ref{thm; blow-up_KS}, 
we will establish an ordinary differential inequality for $\phi_U+\phi_V$. 
As a preparation of the proof, 
we show pointwise upper estimates for $u$ and $v$.

%================================================================%
%                                                Lemma 4.2                                               %
%================================================================%
\begin{lem}\label{lemprofile}
For all $L>0$, $M_0>0$ and $\ep>0$ there exist 
$C_1>0$ and $C_2>0$ such that
if $u_0, v_0$ satisfy \eqref{initial}, \eqref{initial1} and 
$\int_\Omega( u_0(x) + v_0(x) )\,dx = M_0$ 
as well as \eqref{initial_profile}, then 
  \begin{align}\label{profile}
    u(r,t) \le C_1 r^{-n(n-1)-\ep}
    \quad \mbox{and} \quad 
    v(r,t) \le C_2 r^{-n(n-1)-\ep}
  \end{align}
for all $r \in (0,R)$ and $t \in (0,\min\{1,T_{\rm max}\})$.
\end{lem}
%------------------------------proof------------------------------%
\begin{proof}
We put $\tilde{u}(x,t) := e^{-\mu_1 t}u(x,t)$. 
Then we see from \eqref{P1} that 
  \begin{align}\label{profilecondi1}
    \begin{cases}
      \tilde{u}_t \le d_1 \Delta \tilde{u} - \chi_1 \nabla \cdot (\tilde{u} \nabla w), 
        &\quad x \in \Omega, t>0,\\
      \nabla \tilde{u} \cdot \nu = \nabla w \cdot \nu =0,
        &\quad x \in \partial\Omega, t>0,\\
      \tilde{u}(x,0) = u_0(x),
        &\quad x \in \Omega.
    \end{cases}
  \end{align}
Also, we have $\int_\Omega \tilde{u}(\cdot,0)\,dx \le M_0$. 
Next we verify that 
  \begin{align}\label{profilecondi2}
    \int_\Omega |x|^{(n-1)\theta} |\nabla w(x,t)|^\theta \,dx \le c_1
  \end{align}
for all $t \in (0,\min\{1,T_{\rm max}\})$ with some $\theta>n$ and $c_1>0$.
Since the third equation in \eqref{P1} yields 
$\int_\Omega (\alpha u + \beta v)\,dx = \gamma \int_\Omega w\,dx$, 
we can observe from the third equation in \eqref{P1} that
  \[
    |r^{n-1} w_r| 
    = \left| 
         - \frac{1}{d_3}  \int^r_0 \rho^{n-1} (\alpha u + \beta v - \gamma w)\,d\rho 
      \right|
    \le \frac{2}{d_3} \int^{R}_0 \rho^{n-1} (\alpha u + \beta v)\,d\rho.
  \]
Estimating the right-hand side by Lemma \ref{L1esti} and 
the condition $\int_\Omega ( u_0 + v_0 )\,dx = M_0$, we see that
for all $r \in (0,R)$ and $t \in (0,\min\{1,T_{\rm max}\})$,
  \begin{align}\label{rn-1wr}
    |r^{n-1} w_r| \le \frac{2 (\alpha e^{\mu_1} + \beta e^{\mu_2}) M_0}{d_3 \omega_n}=:c_2, 
  \end{align}
where $\omega_n:=n|B_1(0)|$.
Given $\ep>0$, let us choose $\theta>n$ so large such that 
  \begin{align}\label{profilecondi3}
    n(n-1)+\ep>\frac{n(n-1)\theta}{\theta-n}.
  \end{align}
Then, due to \eqref{rn-1wr} we can derive 
\eqref{profilecondi2} with $c_1=c_2^\theta|\Omega|$. 
Thanks to \eqref{initial_profile}, \eqref{profilecondi1}, 
\eqref{profilecondi2} and \eqref{profilecondi3}, 
we can apply \cite[Theorem 1.1]{F-2020_profiles} 
to find $c_3>0$ such that $\tilde{u}(x,t) \le c_3|x|^{-n(n-1)-\ep}$ 
for all $x \in \Omega$ and $t \in (0,\min\{1,T_{\rm max}\})$, 
which means the first claim in \eqref{profile}. 
Similarly, the second claim in \eqref{profile} can be obtained.
\end{proof}

We next derive differential inequalities for $\phi_U$ and $\phi_V$. 
The following lemma is based on \cite[Lemmas 4.1, 4.3 and 4.5]{W-2018}.

%================================================================%
%                                                Lemma 4.3                                               %
%================================================================%
\begin{lem}\label{KSlem4.3}
Let $s_0 \in (0,R^n)$ and 
let $b \in \big( 1-\frac{2}{n}, \min\{1,2-\frac{4}{n}\} \big)$
satisfy
  \begin{align}\label{kappa1kappa2}
    (n-1)(\max\{\kappa_1,\kappa_2\}-1)<\frac{b}{2}.
  \end{align}
Then, for all $L>0$, $M_0>0$ and $\ep>0$ there exist 
$C_1>0$, $C_2>0$, $C_3>0$ and $C_4>0$ such that
if $u_0, v_0$ satisfy \eqref{initial}, \eqref{initial1} and 
$\int_\Omega( u_0(x) + v_0(x) )\,dx = M_0$ 
as well as \eqref{initial_profile},
then 
  \begin{align}\label{basic_phi_U'_KS}
    \phi_U'(t) &\ge - C_1 s_0^{\frac{3-b}{2}-\frac{2}{n}} \sqrt{\psi_U(t)}
  \notag \\ 
  &\quad\,
                          + \frac{\alpha \chi_1 n}{d_3} \psi_U(t)
                          - (b+1)\frac{\gamma \chi_1 n}{d_3} s_0 \int^{s_0}_0 s^{-b-1} U(s,t)W(s,t)\,ds 
  \notag \\ 
  &\quad\,
                          - C_2 s_0^{-(n-1)(\kappa_1-1)+\frac{3-b}{2}-\ep} \sqrt{\psi_U(t)} 
  \notag \\ 
  &\quad\,
                          - a_1 \mu_1 n^{\lambda_1-1} 
                               \int^{s_0}_0 s^{-b}(s_0-s) 
                                 \left( \int^s_0 U_s(\sigma,t)V_s^{\lambda_1-1}(\sigma,t)\,d\sigma \right)\,ds
  \end{align}
and
  \begin{align}\label{basic_phi_V'_KS}
    \phi_V'(t) &\ge - C_3 s_0^{\frac{3-b}{2}-\frac{2}{n}} \sqrt{\psi_V(t)}
  \notag \\ 
  &\quad\,
                          + \frac{\beta \chi_2 n}{d_3} \psi_V(t)
                          - (b+1)\frac{\gamma \chi_2 n}{d_3} s_0 \int^{s_0}_0 s^{-b-1} V(s,t)W(s,t)\,ds 
  \notag \\ 
  &\quad\,
                          - C_4 s_0^{-(n-1)(\kappa_2-1)+\frac{3-b}{2}-\ep} \sqrt{\psi_V(t)} 
  \notag \\ 
  &\quad\,
                          - a_2 \mu_2 n^{\lambda_2-1} 
                               \int^{s_0}_0 s^{-b}(s_0-s) 
                                 \left( \int^s_0 V_s(\sigma,t)U_s^{\lambda_2-1}(\sigma,t)\,d\sigma \right)\,ds
  \end{align}
for all $t \in (0,\min\{1,T_{\rm max}\})$.
\end{lem}
%------------------------------proof------------------------------%
\begin{proof} 
We first derive a differential inequality for $\phi_U$ 
from the first and third equations in \eqref{P1}. 
The first equation in \eqref{P1} is rewritten as 
  \begin{align}\label{ut} 
    u_t = d_1 r^{1-n}(r^{n-1} u_r)_r 
            - \chi_1 r^{1-n}(r^{n-1} u w_r)_r 
            + \mu_1 u (1 - u^{\kappa_1-1} - a_1 v^{\lambda_1-1}) 
  \end{align} 
for all $r \in (0,R)$ and $t \in (0,T_{\rm max})$. 
Also, we multiply the third equation in \eqref{P1} by $r^{n-1}$ 
and integrate it over $[0,s^{\frac{1}{n}}]$
to obtain
  \begin{align}\label{thirdeqKS}
    s^{1-\frac{1}{n}} w_r(s^{\frac{1}{n}},t) 
    = - \frac{\alpha}{d_3} U(s,t)
       - \frac{\beta}{d_3} V(s,t) 
       + \frac{\gamma}{d_3} W(s,t). 
  \end{align}
Thus, noticing that
$U_s(s,t) = \frac{1}{n} u(s^{\frac{1}{n}},t)$, 
$V_s(s,t) = \frac{1}{n} v(s^{\frac{1}{n}},t)$,
$W_s(s,t) = \frac{1}{n} w(s^{\frac{1}{n}},t)$
and
$U_{ss}(s,t) = \frac{1}{n^2} s^{\frac{1}{n}-1} u_r(s^{\frac{1}{n}},t)$,
we can verify from \eqref{ut} and \eqref{thirdeqKS} that
  \begin{align*}
    U_t &= d_1 n^2 s^{2-\frac{2}{n}} U_{ss} 
              + \frac{\alpha \chi_1 n}{d_3} UU_s 
              + \frac{\beta \chi_1 n}{d_3} U_sV 
              - \frac{\gamma \chi_1 n}{d_3} U_sW
  \\ 
  &\quad\,
              + \mu_1 U 
              - \mu_1 n^{\kappa_1-1} \int^s_0 U_s^{\kappa_1}(\sigma,t)\,d\sigma 
              - a_1 \mu_1 n^{\lambda_1-1} \int^s_0 U_s(\sigma,t)V_s^{\lambda_1-1}(\sigma,t)\,d\sigma
  \end{align*}
for all $s \in (0,R^n)$ and $t \in (0,T_{\rm max})$. 
In light of this identity we infer that 
  \begin{align}\label{dif_phi_U}
    \phi_U'(t) &\ge d_1 n^2 \int^{s_0}_0 s^{2-\frac{2}{n}-b}(s_0-s) U_{ss}\,ds
  \notag \\ 
  &\quad\,
                          + \frac{\alpha \chi_1 n}{d_3} 
                             \psi_U(t)
                          - \frac{\gamma \chi_1 n}{d_3}  \int^{s_0}_0 s^{-b}(s_0-s) U_sW\,ds
  \notag \\ 
  &\quad\,
                          - \mu_1 n^{\kappa_1-1}
                               \int^{s_0}_0 s^{-b}(s_0-s) 
                                 \left( \int^s_0 U_s^{\kappa_1}(\sigma,t)\,d\sigma \right)\,ds
  \notag \\ 
  &\quad\,
                          - a_1 \mu_1 n^{\lambda_1-1} 
                               \int^{s_0}_0 s^{-b}(s_0-s) 
                                 \left( \int^s_0 U_s(\sigma,t)V_s^{\lambda_1-1}(\sigma,t)\,d\sigma \right)\,ds
  \notag \\ 
                  &=: I_1 + I_2 + I_3 + I_4 + I_5 
  \end{align}
for all $t \in (0,T_{\rm max})$.
We next show estimates for $I_1, I_3$ and $I_4$. 
By arguments similar to those in \cite[estimate (4.4) and Lemma 4.3]{W-2018}, 
it follows that 
  \begin{align}\label{firstterm}
    I_1 \ge - c_1 s_0^{\frac{3-b}{2}-\frac{2}{n}} \sqrt{\psi_U(t)}
  \end{align}
for all $t \in (0,T_{\rm max})$ with some $c_1>0$.
As to $I_4$, by Fubini's theorem we deduce that
  \begin{align}\label{fubini}
      I_4 &= - \mu_1 n^{\kappa_1-1} 
                  \int^{s_0}_0 \left( \int^{s_0}_\sigma s^{-b}(s_0-s)\,ds \right)
                  U_s^{\kappa_1}(\sigma,t)\,d\sigma
    \notag \\
    &\ge - \mu_1 n^{\kappa_1-1}
               \int^{s_0}_0 \left( \int^{s_0}_\sigma s^{-b}\,ds \right) 
               (s_0-\sigma) U_s^{\kappa_1}(\sigma,t)\,d\sigma
    \notag \\
    &\ge - \frac{\mu_1 n^{\kappa_1-1}}{1-b}
               s_0^{1-b} \int^{s_0}_0 (s_0-s) U_s^{\kappa_1}(s,t)\,ds, 
  \end{align}
so that by virtue of \eqref{kappa1kappa2} 
we apply \cite[Lemma 4.5]{W-2018} to find $c_2>0$ such that 
  \begin{align}\label{fifthterm}
    I_4 \ge - c_2 s_0^{-(n-1)(\kappa_1-1)+\frac{3-b}{2}-\ep} \sqrt{\psi_U(t)} 
  \end{align}
for all $t \in (0,\min\{1,T_{\rm max}\})$.
With regard to $I_3$, as in \cite[estimate (4.5)]{W-2018}, 
by integration by parts we can make sure that
  \begin{align}\label{fourthterm} 
    I_3 \ge - (b+1)\frac{\gamma \chi_1 n}{d_3} s_0 \int^{s_0}_0 s^{-b-1} UW\,ds
  \end{align}
for all $t \in (0,T_{\rm max})$. 
By \eqref{dif_phi_U}, \eqref{firstterm}, 
\eqref{fifthterm} and \eqref{fourthterm} we arrive at \eqref{basic_phi_U'_KS}.
The inequality \eqref{basic_phi_V'_KS} is derived by a similar argument.
\end{proof}

We next present an estimate for the integrals involving $W$ 
in \eqref{basic_phi_U'_KS} and \eqref{basic_phi_V'_KS}. 
Since the following lemma has been obtained in \cite[Lemma 4.4]{TQ_2020}, 
we provide only the statement of the lemma. 

%================================================================%
%                                                Lemma 4.4                                               %
%================================================================%
\begin{lem}
Let $s_0 \in (0,R^n)$ and $b \in (0,\min\{1,2-\frac{4}{n}\})$. 
For all $L>0$ and $M_0>0$ there exists $C>0$ such that
if $u_0, v_0$ satisfy \eqref{initial}, \eqref{initial1} and 
$\int_\Omega( u_0(x) + v_0(x) )\,dx = M_0$ 
as well as \eqref{initial_profile},
then 
  \begin{align*}
    &- (b+1)\frac{\gamma \chi_1 n}{d_3} s_0 \int^{s_0}_0 s^{-b-1} U(s,t)W(s,t)\,ds
      - (b+1)\frac{\gamma \chi_2 n}{d_3} s_0 \int^{s_0}_0 s^{-b-1} V(s,t)W(s,t)\,ds
  \notag \\
  &\quad\,
    \ge - C s_0^{1-b+\frac{2}{n}} 
          - C s_0^{\frac{2}{n}} (\psi_U(t) + \psi_V(t))
  \end{align*} 
for all $t \in (0,\min\{1,T_{\rm max}\})$. 
\end{lem} 

We estimate the last terms 
in \eqref{basic_phi_U'_KS} and \eqref{basic_phi_V'_KS}. 

%================================================================%
%                                                Lemma 4.5                                               %
%================================================================%
\begin{lem}\label{KSlem4.5}
Let $s_0 \in (0,R^n)$ and let $b \in (0,1)$ 
satisfy
  \begin{align*}
    (n-1)(\max\{\lambda_1,\lambda_2\}-1)<\frac{b}{2}.
  \end{align*}
Then, for all $L>0$, $M_0>0$ and $\ep>0$ there exist 
$C_1>0$ and $C_2>0$ such that
if $u_0, v_0$ satisfy \eqref{initial}, \eqref{initial1} and 
$\int_\Omega( u_0(x) + v_0(x) )\,dx = M_0$ 
as well as \eqref{initial_profile},
then 
  \begin{align}\label{UsVs1}
     &- a_1 \mu_1 n^{\lambda_1-1} \int^{s_0}_0 s^{-b}(s_0-s) 
        \left( \int^s_0 U_s(\sigma,t)V_s^{\lambda_1-1}(\sigma,t)\,d\sigma \right)\,ds
  \notag \\ 
  &\quad\,
    \ge - C_1 s_0^{-(n-1)(\lambda_1-1)+\frac{3-b}{2}-\ep} \sqrt{\psi_U(t)}
  \end{align}
and 
  \begin{align}\label{UsVs2}
    &- a_2 \mu_2 n^{\lambda_2-1} \int^{s_0}_0 s^{-b}(s_0-s) 
       \left( \int^s_0 V_s(\sigma,t)U_s^{\lambda_2-1}(\sigma,t)\,d\sigma \right)\,ds
  \notag \\ 
  &\quad\,
    \ge - C_2 s_0^{-(n-1)(\lambda_2-1)+\frac{3-b}{2}-\ep} \sqrt{\psi_V(t)}
  \end{align}
for all $t \in (0,\min\{1,T_{\rm max}\})$.
\end{lem}
%------------------------------proof------------------------------%
\begin{proof}
It is sufficient to show \eqref{UsVs1}, 
because we can confirm \eqref{UsVs2} by a similar argument. 
Given $\ep>0$, let us fix $\eta>0$ fulfilling 
  \[
    \frac{\eta}{n}(\lambda_1-1) \le \min\{\ep,1\}
    \quad \mbox{and} \quad
    (n-1)(\lambda_1-1)+\frac{\eta}{n}(\lambda_1-1) < \frac{b}{2}.
  \]
By an argument similar to that in \eqref{fubini}, we infer from Fubini's theorem that 
  \begin{align}\label{UsVslam}
    &- a_1 \mu_1 n^{\lambda_1-1} \int^{s_0}_0 s^{-b}(s_0-s) 
       \left( \int^s_0 U_s(\sigma,t)V_s^{\lambda_1-1}(\sigma,t)\,d\sigma \right)\,ds
    \notag \\
    &\quad\,
    \ge - \frac{a_1 \mu_1 n^{\lambda_1-1}}{1-b} 
             s_0^{1-b} \int^{s_0}_0 (s_0-s) U_s V_s^{\lambda_1-1}\,ds.
  \end{align}
Moreover, from Lemma \ref{lemprofile} there exists $c_1>0$ such that
  \begin{align*}
    V_s^{\lambda_1-1}(s,t) 
    = \left( \frac{1}{n}v(s^\frac{1}{n},t) \right)^{\lambda_1-1} 
    \le c_1 s^{-(n-1)(\lambda_1-1)-\frac{\eta}{n}(\lambda_1-1)}
  \end{align*}
for all $s \in (0,R^n)$ and $t \in (0,\min\{1,T_{\rm max}\})$, 
and thus we derive that 
  \[
    s_0^{1-b} \int^{s_0}_0 (s_0-s) U_s V_s^{\lambda_1-1}\,ds
    \le c_1 s_0^{1-b} \int^{s_0}_0 
                 s^{-(n-1)(\lambda_1-1)-\frac{\eta}{n}(\lambda_1-1)}(s_0-s) U_s \,ds
  \]
for all $t \in (0,\min\{1,T_{\rm max}\})$. 
Combining this inequality with \eqref{UsVslam} and using 
an argument similar to that in the proof of \cite[Lemma 4.5]{W-2018}, 
we can verify that \eqref{UsVs1} holds. 
\end{proof}

We can now pass to the proof of Theorem \ref{thm; blow-up_KS}.

\begin{proof}[Proof of Theorem \ref{thm; blow-up_KS}] 
Let us show that there are 
$b \in \big(1-\frac{2}{n},\min\{1,2-\frac{4}{n}\}\big)$, 
$s_1\in(0,R^n)$, $c_1>0$ and $c_2>0$ such that for any $s_0\in(0,s_1)$, 
  \begin{align}\label{phiUphiV}
    \phi_U'(t) + \phi_V'(t) 
    \ge c_1 s_0^{-(3-b)} (\phi_U(t) + \phi_V(t))^2 - c_2 s_0^{1-b+\frac{2}{n}} 
  \end{align}
for all $t \in (0,\min\{1,T_{\rm max}\})$. 
By virtue of \eqref{KSblowup}, in the case $n=3$ we can make sure that
$(n-1)(\max\{\kappa_1,\lambda_1,\kappa_2,\lambda_2\}-1)<\frac{1}{3}=\frac{2-\frac{4}{n}}{2}$, 
whereas in the case $n\ge4$ we have
$(n-1)(\max\{\kappa_1,\lambda_1,\kappa_2,\lambda_2\}-1)<\frac{1}{2}$. 
Thus we take $b \in \big(1-\frac{2}{n},\min\{1,2-\frac{4}{n}\}\big)$ such that
  \begin{align}\label{bcondi}
    (n-1)(\max\{\kappa_1,\lambda_1,\kappa_2,\lambda_2\}-1)<\frac{b}{2}.
  \end{align}
Also, let us pick $\ep>0$ fulfilling
  \begin{align}\label{epcondi}
    2\ep \le 1-\frac{2}{n}.
  \end{align}
Then we can apply Lemmas \ref{KSlem4.3}--\ref{KSlem4.5} 
to find $c_3>0$ and $c_4>0$ such that
  \begin{align*}
    \phi_U'(t) + \phi_V'(t) 
    &\ge - c_3 s_0^{\frac{3-b}{2}-\frac{2}{n}} (\sqrt{\psi_U(t)} + \sqrt{\psi_V(t)})
  \notag \\ 
  &\quad\,
                          + c_4 (\psi_U(t) + \psi_V(t))
  \notag \\ 
  &\quad\,
                          - c_3 s_0^{1-b+\frac{2}{n}} 
                          - c_3 s_0^{\frac{2}{n}} (\psi_U(t) + \psi_V(t)) 
  \notag \\ 
  &\quad\,
                          - c_3 s_0^{-(n-1)(\kappa_1-1)+\frac{3-b}{2}-\ep} \sqrt{\psi_U(t)}
                          - c_3 s_0^{-(n-1)(\kappa_2-1)+\frac{3-b}{2}-\ep} \sqrt{\psi_V(t)}
  \notag \\ 
  &\quad\,
                         - c_3 s_0^{-(n-1)(\lambda_1-1)+\frac{3-b}{2}-\ep} \sqrt{\psi_U(t)}
                         - c_3 s_0^{-(n-1)(\lambda_2-1)+\frac{3-b}{2}-\ep} \sqrt{\psi_V(t)}
  \end{align*}
for all $t \in (0,\min\{1,T_{\rm max}\})$. 
Moreover, Young's inequality yields that  
there exists $c_5>0$ such that
  \begin{align*}
    \phi_U'(t) + \phi_V'(t) 
    &\ge \frac{3}{4} c_4 (\psi_U(t) + \psi_V(t)) 
            - c_3 s_0^{\frac{2}{n}} (\psi_U(t) + \psi_V(t))
  \notag \\
  &\quad\, 
            - c_5 s_0^{1-b+\frac{2}{n}}
               \left( s_0^{2-\frac{6}{n}} 
                       + 1 
                       + s_0^{2-\frac{2}{n}-2(n-1)(\kappa_1-1)-2\ep}
                       + s_0^{2-\frac{2}{n}-2(n-1)(\kappa_2-1)-2\ep}
               \right.
  \notag \\
  &\quad\,\quad\,\quad\,\quad\,\quad\,\quad\,\quad\,
               \left.
                       + s_0^{2-\frac{2}{n}-2(n-1)(\lambda_1-1)-2\ep}
                       + s_0^{2-\frac{2}{n}-2(n-1)(\lambda_2-1)-2\ep}\right)
  \end{align*}
for all $t \in (0,\min\{1,T_{\rm max}\})$. 
Here, we choose $s_1 \in (0,R^n)$ such that
$c_3s_1^{\frac{2}{n}} \le \frac{c_4}{4}$, 
and then we see that $c_3s_0^{\frac{2}{n}} \le \frac{c_4}{4}$ for all $s_0 \in (0,s_1)$. 
Also, since it follows from \eqref{bcondi}, \eqref{epcondi} and 
the condition $b<2-\frac{4}{n}$ that
  \begin{align*}
    &2-\frac{2}{n}-2(n-1)(\max\{\kappa_1,\lambda_1,\kappa_2,\lambda_2\}-1)-2\ep 
  \\ 
  &\quad\, 
    \ge 1 - 2(n-1)(\max\{\kappa_1,\lambda_1,\kappa_2,\lambda_2\}-1)
  \\ 
  &\quad\, 
    > 1 - b > 0,             
  \end{align*}
we infer from the relation $s_0 < R^n$ that  
  \begin{align}\label{phiUphiV2}
    \phi_U'(t) + \phi_V'(t) 
    \ge \frac{c_4}{2} (\psi_U(t) + \psi_V(t))
          - c_6 s_0^{1-b+\frac{2}{n}}
  \end{align}
for all $s_0 \in (0,s_1)$ and $t \in (0,\min\{1,T_{\rm max}\})$,
where 
  \begin{align*}
    c_6&:=c_5(R^{2n-6} 
                  + 1 
                  + R^{2n-2-2n(n-1)(\kappa_1-1)-2n\ep}
                  + R^{2n-2-2n(n-1)(\kappa_2-1)-2n\ep}\\
    &\quad\,\quad\,\,\,
                  + R^{2n-2-2n(n-1)(\lambda_1-1)-2n\ep}
                  + R^{2n-2-2n(n-1)(\lambda_2-1)-2n\ep}).
  \end{align*}
Since \cite[Lemma 4.4]{W-2018} provides $c_7>0$ such that
  \[
    \psi_U(t) \ge c_7 s_0^{-(3-b)} \phi_U^2(t)
    \quad \mbox{and} \quad 
    \psi_V(t) \ge c_7 s_0^{-(3-b)} \phi_V^2(t)
  \]
for all $t \in (0,T_{\rm max})$, 
these inequalities together with \eqref{phiUphiV2} lead to \eqref{phiUphiV}. 

Now let us prove \eqref{blowupKS}. 
We take $s_0 \in (0,s_1)$ so small such that
  \[
    \sqrt{\frac{c_2}{c_1}} s_0^{\frac{1}{n}} 
          +\frac{2}{c_1}s_0
    \le
    \frac{\widetilde{M}_0}{2^{3-b} \omega_n}, 
  \]
and then the relation
  \[
    \frac{\widetilde{M}_0}{2^{3-b} \omega_n} s_0^{2-b} 
    \ge \sqrt{\frac{c_2}{c_1}} s_0^{2-b+\frac{1}{n}} 
          +\frac{2}{c_1}s_0^{3-b}
  \]
holds.
Let us put $r_\star:=\left( \frac{s_0}{4} \right)^\frac{1}{n} \in (0,R)$.
Moreover, we suppose that the initial data $u_0,v_0$ satisfy  
\eqref{initial}, \eqref{initial1}, \eqref{initial KS} and \eqref{initial_profile}. 
Then we can make sure that
$\phi_U(0)+\phi_V(0) \ge \frac{\widetilde{M}_0}{2^{3-b} \omega_n} s_0^{2-b}$ 
by \cite[estimate (5.5)]{W-2018}. 
Thus, as in the proof of \cite[Lemma 4.6]{F_2021_optimal} 
(with $d_1(s_0)=c_1s_0^{-(3-b)}$, $d_2(s_0)=c_2s_0^{1-b-\frac{2}{n}}$ 
and $\phi(s_0)=\frac{\widetilde{M}_0}{2^{3-b} \omega_n}s_0^{2-b}$), 
we can establish that $T_{\rm max} \le \frac{1}{2}$, 
and hence \eqref{blowupKS} results from \eqref{criterion}.
\end{proof}

%================================================================%
%================================================================%
%                                                Section 5                                                %
%================================================================%
%================================================================%
\section{Finite-time blow-up in a model of J\"{a}ger--Luckhaus type}\label{sectionJL}

In this section we deal with 
the system \eqref{P} 
with $\lambda_1 = \lambda_2 = 2$ 
in the case that $h$ is of J\"{a}ger--Luckhaus type. 
Namely, we consider the system 
  \begin{align}\label{P2}
    \begin{cases}
      u_t = d_1 \Delta u 
                - \chi_1 \nabla \cdot (u \nabla w) 
                + \mu_1 u (1- u^{\kappa_1-1} - a_1 v),
        &\quad x \in \Omega,\ t>0,\\
      v_t = d_2 \Delta v 
                - \chi_2 \nabla \cdot (v \nabla w) 
                + \mu_2 v (1- a_2 u - v^{\kappa_2-1}),
        &\quad x \in \Omega,\ t>0,\\
      0 = d_3 \Delta w + \alpha u + \beta v 
            - \frac{1}{|\Omega|} \int_\Omega (\alpha u + \beta v)\,dx,
        &\quad x \in \Omega,\ t>0,\\
      \nabla u \cdot \nu = \nabla v \cdot \nu = \nabla w \cdot \nu = 0, 
        &\quad x \in \partial \Omega,\ t>0,\\
      u(x,0) = u_0(x), \ 
      v(x,0) = v_0(x), 
        &\quad x \in \Omega, 
    \end{cases}
  \end{align}
with
  \[
    \int_\Omega w\,dx = 0, 
      \quad t>0,
  \]
where $\Omega = B_R(0) \subset \mathbb{R}^n$ $(n\ge5)$ 
is a ball with some $R>0$, and 
$u_0,v_0$ satisfy \eqref{initial} and 
  \begin{align}\label{initial2}
    u_0, v_0 \ \mbox{are radially symmetric 
                            and nonincreasing with respect to}\ |x|. 
  \end{align}
Let $(u,v,w)$ be the local classical solution of \eqref{P2} 
on $[0,T_{\rm max})$
given in Lemma \ref{localsol}. 
Moreover, we put
  \[
    \overline{M}(t) := \frac{1}{|\Omega|} \int_\Omega (\alpha u + \beta v)\,dx.
  \]
The main theorem in this section reads as follows.

%================================================================%
%                                              Theorem 5.1                                              %
%================================================================%
%\newpage
\begin{thm}\label{thm; blow-up_JL}
Let $\Omega = B_R(0) \subset \mathbb{R}^n$ $(n\ge5)$ 
and
let $d_1, d_2, d_3, \chi_1, \chi_2, a_1, a_2, \alpha, \beta>0$ 
and $\kappa_1, \kappa_2 \in (1,2]$.
Assume that $\mu_1$ and $\mu_2$ satisfy that 
      \begin{align}\label{JLblowup}
        \chi_1 > 
          \begin{cases}
            \frac{a_1 d_3 \mu_1}{\beta} \cdot \frac{n}{n-4}
              &\mbox{if}\ \kappa_1<2,\\
            \max \big\{ \frac{d_3 \mu_1}{\alpha}, \frac{a_1 d_3 \mu_1}{\beta} \big\}
              \cdot \frac{n}{n-4}
              &\mbox{if}\ \kappa_1=2
          \end{cases}
          \quad \mbox{and} \quad 
        \chi_2 > \frac{a_2 d_3 \mu_2}{\alpha}.
      \end{align}
    Then, for all $M_0>0$ and $\widetilde{M}_0 \in (0,M_0)$ 
    there exists $r_\star \in (0,R)$ with the following property\/{\rm :}
    If $u_0, v_0$ satisfy \eqref{initial} and \eqref{initial2} 
    as well as
      \begin{align}\label{initial JL} 
        \int_\Omega u_0(x)\,dx = M_0 
        \quad \mbox{and} \quad 
        \int_{B_{r_\star}(0)} u_0(x)\,dx \ge \widetilde{M}_0, 
      \end{align}
    then the corresponding solution $(u,v,w)$ of \eqref{P2} 
    blows up in finite time 
    in the sense that $T_{\rm max}<\infty$ and
      \begin{align*}
        \lim_{t \nearrow T_{\rm max}} 
        ( \| u(\cdot,t) \|_{L^\infty(\Omega)} + \| v(\cdot,t) \|_{L^\infty(\Omega)} )
        = \infty. 
      \end{align*}
\end{thm}
%%%%%%%%%%%%%%%%%%%%%%%%%%%%%%%%%%%%%%%%%%%%%%%%%%%%%%%%%%%%%%%%%%
%%%%%%%%%%%%%%%%%%%%%%%%%%%%%%%%%%%%%%%%%%%%%%%%%%%%%%%%%%%%%%%%%%

\begin{remark}
If $v=0$, $d_1=d_3=1$, $\chi_1=1$, $\kappa_1=2$ and $\alpha=1$ in \eqref{P2}, 
then the condition \eqref{JLblowup} is the same as in \cite[(1.4b)]{F_2021_optimal}.
Thus this theorem is a generalization of \cite[Theorem 1.1]{F_2021_optimal} 
in the higher-dimensional case.
\end{remark}

\begin{remark}
If $\mu_1$ and $\mu_2$ fulfill that 
  \begin{align*}
       \chi_1 > \frac{a_1 d_3 \mu_1}{\beta}
          \quad \mbox{and} \quad 
        \chi_2> 
          \begin{cases}
            \frac{a_2 d_3 \mu_2}{\alpha} \cdot \frac{n}{n-4}
              &\mbox{if}\ \kappa_2<2,\\
            \max \big\{ \frac{d_3 \mu_2}{\beta}, \frac{a_2 d_3 \mu_2}{\alpha} \big\}
              \cdot \frac{n}{n-4}
              &\mbox{if}\ \kappa_2=2, 
          \end{cases}
  \end{align*}
then, by replacing $u_0$ with $v_0$ in \eqref{initial JL}, 
a similar result can be derived.
\end{remark}
Now, let us define the functions $U,V$ and $W$ as in Section \ref{sectionKS}. 
Furthermore, we introduce the functions $\phi_U, \phi_V, \psi_U$ and $\psi_V$ 
defined in Section \ref{sectionKS} with $b \in (1,2)$.

\begin{remark}
As to the functions $\phi_U$ and $\phi_V$, even though 
we have taken $b \in (1,2)$ unlike e.g.\ 
\cite{B-F-L, W-2018_nonlinear, W-2018} and Section \ref{sectionKS}, 
we see that 
these functions are well-defined and that
$\phi_U$ and $\phi_V$ belong to 
$C^0([0,T_{\rm max})) \cap C^1((0,T_{\rm max}))$, 
since we can confirm from the following lemma that 
$s^{-b}(s_0-s)U(s,t) \le C s^{1-b}(s_0-s)$ for all $t \in [0,T_{\rm max})$ 
with some $C>0$
and that for any $t_0 \in (0,T_{\rm max})$ and $t_1 \in (t_0,T_{\rm max})$, 
$|\frac{d}{dt} s^{-b}(s_0-s)U(s,t)| \le \widetilde{C} s^{1-b}(s_0-s)$
for all $t \in (t_0,t_1)$ with some $\widetilde{C}>0$ 
(for details see the proof of  \cite[Lemma 4.1]{F_2021_optimal}).
\end{remark}

We first show concavity of $U$ and $V$, 
which is obtained by the comparison principle.

%================================================================%
%                                                Lemma 5.2                                               %
%================================================================%
\begin{lem}\label{increase}
Assume that \eqref{initial} and \eqref{initial2} hold. 
If $\mu_1$ and $\mu_2$ satisfy that 
  \begin{align}\label{mu1mu2}
    0 < \mu_1 < \frac{\beta \chi_1}{a_1 d_3}
    \quad \mbox{and} \quad
    0 < \mu_2 < \frac{\alpha \chi_2}{a_2 d_3},
  \end{align}
then 
  \begin{align}\label{UssVss} 
    U_{ss}(s,t) \le 0 
      \quad \mbox{and} \quad 
    V_{ss}(s,t) \le 0 
  \end{align}
and moreover,
  \begin{align}\label{mean}
    U_s(s,t) \le \frac{U(s,t)}{s} \le U_s(0,t)
      \quad \mbox{and} \quad 
    V_s(s,t) \le \frac{V(s,t)}{s} \le V_s(0,t)
  \end{align}
for all $s \in (0,R^n)$ and $t \in [0,T_{\rm max})$. 
\end{lem}
%------------------------------proof------------------------------%
\begin{proof} 
Since \eqref{mean} is proved  
by using the mean value theorem 
and \eqref{UssVss}, we need only to show \eqref{UssVss}.
To see this, we establish that $u_r \le 0$ and $v_r \le 0$ 
for all $r \in (0,R)$ and $t \in (0,T_{\rm max})$ 
because we have $U_{ss}(s,t) = \frac{1}{n^2} s^{\frac{1}{n}-1} u_r(s^\frac{1}{n},t)$ 
and $V_{ss}(s,t) = \frac{1}{n^2} s^{\frac{1}{n}-1} v_r(s^\frac{1}{n},t)$
for all $s \in (0,R^n)$ and $t \in (0,T_{\rm max})$. 
By an approximation argument 
as in Step 2 in the proof of \cite[Lemma 2.2]{W-2018_nonlinear},  
we can assume without loss of generality that 
$u_0, v_0 \in C^2(\overline{\Omega})$ with 
$\nabla u_0 \cdot \nu = \nabla v_0 \cdot \nu = 0$. 
Then this assumption guarantees that
  \[
    u_r, v_r \in C^0([0,R] \times [0,T_{\rm max})) 
                   \cap 
                   C^{2,1}((0,R) \times (0,T_{\rm max})).
  \]
From the third equation in \eqref{P2} it follows that
  \begin{align*}
    \chi_1 r^{1-n}(r^{n-1} u w_r)_r
    &= \chi_1 u_r w_r + \chi_1 u \cdot r^{1-n}(r^{n-1} w_r)_r 
    \\
    &= \chi_1 u_r w_r 
         + \chi_1 u \cdot \left( - \frac{\alpha}{d_3} u 
                                         - \frac{\beta}{d_3} v 
                                         + \frac{1}{d_3} \overline{M}(t) 
                                 \right)
    \\
    &= \chi_1 u_r w_r
         - \frac{\alpha \chi_1}{d_3} u^2 
         - \frac{\beta \chi_1}{d_3} uv 
         + \frac{\chi_1}{d_3} \overline{M}(t) u, 
  \end{align*}
and so we obtain
  \begin{align*}
    u_{rt} &= (d_1 r^{1-n}(r^{n-1} u_r)_r 
                 - \chi_1 r^{1-n}(r^{n-1} u w_r)_r
                 + \mu_1 u (1 - u^{\kappa_1-1} - a_1 v))_r
    \notag \\ 
           &= d_1 u_{rrr} + d_1 \frac{n-1}{r} u_{rr} - d_1 \frac{n-1}{r^2} u_{r}
    \notag \\
    &\quad\,
                - \chi_1 u_{rr} w_r - \chi_1 u_{r} w_{rr}
                + \frac{2 \alpha \chi_1}{d_3} u u_r 
                + \frac{\beta \chi_1}{d_3} u_r v + \frac{\beta \chi_1}{d_3} u v_r 
                - \frac{\chi_1}{d_3} \overline{M}(t) u_r
    \notag \\
    &\quad\,
                + \mu_1 u_r
                - \mu_1 \kappa_1 u^{\kappa_1-1} u_r
                - a_1 \mu_1 u_r v
                - a_1 \mu_1 u v_r
  \end{align*}
for all $r \in (0,R)$ and $t \in (0,T_{\rm max})$. 
Moreover, this identity is rewritten as
  \begin{align}\label{urt}
    u_{rt} = d_1 u_{rrr} + A_1(r,t) u_{rr} + B_1(r,t) u_r + C_1(r,t) v_r,
  \end{align}
where
  \begin{align*}
    A_1(r,t) &:= d_1 \frac{n-1}{r} - \chi_1 w_r(r,t), 
    \\
    B_1(r,t) &:= - d_1 \frac{n-1}{r^2} 
                    - \chi_1 w_{rr}(r,t) 
                    + \frac{2 \alpha \chi_1}{d_3} u(r,t) 
                    + \frac{\beta \chi_1}{d_3} v(r,t)
                    - \frac{\chi_1}{d_3} \overline{M}(t)
     \\ &\quad\,\ 
                    + \mu_1
                    - \mu_1 \kappa_1 u^{\kappa_1-1}(r,t) 
                    - a_1 \mu_1 v(r,t) 
\intertext{and}
    C_1(r,t) &:= \frac{\beta \chi_1}{d_3} u(r,t)
                    - a_1 \mu_1 u(r,t) 
  \end{align*}
for $r \in (0,R)$ and $t \in (0,T_{\rm max})$.
Now let us fix $T \in (0,T_{\rm max})$
and denote by $(u_r)_+$ and $(v_r)_+$ 
the positive part functions of $u_r$ and $v_r$, respectively. 
Then we have 
  \begin{align}\label{0R}
    (u_r)_+(0,t) = (u_r)_+(R,t) = 0, 
  \end{align}
because the regularity and radial symmetry of $u$ imply that $(u_r)_+(0,t)=0$ 
and the Neumann boundary condition in \eqref{P2} 
gives $(u_r)_+(R,t)=0$. 
In view of \eqref{urt} 
we can verify from integration by parts and \eqref{0R} that 
  \begin{align}\label{I_1 I_2 I_3}
    &\frac{1}{2} \cdot \frac{d}{dt} \int^R_0 r (u_r)_+^2 \,dr
  \notag \\ 
  &\quad\,
    = \int^R_0 r (u_r)_+ [ d_1 u_{rrr} + A_1(r,t) u_{rr} + B_1(r,t) u_r + C_1(r,t) v_r ]\,dr
  \notag \\ 
  &\quad\,    
    = - d_1 \int^R_0 \left( \frac{1}{2}[(u_r)_+^2]_r + r ( [ (u_r)_+ ]_r )^2 \right)\,dr 
       + \frac{1}{2} \int^R_0 r A_1(r,t) [ (u_r)_+^2 ]_r\,dr 
    \notag \\
    &\quad\,\quad\,
       + \int^R_0 r B_1(r,t) (u_r)_+^2\,dr 
       + \int^R_0 r C_1(r,t) (u_r)_+ v_r \,dr
  \notag \\ 
  &\quad\,
    \le \frac{1}{2} \int^R_0 r A_1(r,t) [ (u_r)_+^2 ]_r\,dr 
       + \int^R_0 r B_1(r,t) (u_r)_+^2\,dr 
       + \int^R_0 r C_1(r,t) (u_r)_+ v_r \,dr
  \notag \\ 
  &\quad\,
    =: J_1 + J_2 + J_3
  \end{align}
for all $t \in (0,T)$. 
Recalling the definition of $A_1$,
we obtain from \eqref{0R} 
and integration by parts that 
  \begin{align*}
   J_1 &= \frac{d_1 (n-1)}{2} \int^R_0 [ (u_r)_+^2 ]_r\,dr
         - \chi_1 \int^R_0 r w_r [ (u_r)_+^2 ]_r\,dr
    \notag \\
    &= \frac{d_1 (n-1)}{2} \left[ (u_r)_+^2(R,t) - (u_r)_+^2(0,t) \right]
    \notag \\
    &\quad\,
         - \chi_1 R w_r(R,t) \cdot (u_r)_+^2(R,t)
         + \chi_1 \int^R_0 w_r (u_r)_+^2\,dr
         + \chi_1 \int^R_0 r w_{rr} (u_r)_+^2\,dr 
    \notag \\
    &\le \chi_1 \int^R_0 w_r (u_r)_+^2\,dr
         + \chi_1 \| w_{rr} \|_{L^\infty((0,R) \times (0,T))} \int^R_0 r (u_r)_+^2\,dr. 
  \end{align*}
Here the third equation in \eqref{P2} yields
  \begin{align}\label{wr}
    w_r = - \frac{1}{d_3} r^{1-n} \int^r_0 \rho^{n-1} (\alpha u + \beta v)\,d\rho 
            + \frac{1}{d_3 n} \overline{M}(t) r 
         \le \frac{1}{d_3 n} \overline{M}(t) r, 
  \end{align}
which together with Lemma \ref{L1esti} means that 
  \begin{align}\label{I_1esti}
   J_1 &\le \frac{\chi_1}{d_3 n} \overline{M}(t) \int^R_0 r (u_r)_+^2\,dr
           + \chi_1 \| w_{rr} \|_{L^\infty((0,R) \times (0,T))} \int^R_0 r (u_r)_+^2\,dr
    \notag \\ 
        &\le   c_1 \int^R_0 r (u_r)_+^2\,dr
  \end{align}
for all $t \in (0,T)$, 
where 
  \[
    c_1 := \frac{\chi_1}{d_3 n |\Omega|}\left(\alpha e^{\mu_1 T} \int_\Omega u_0\,dx 
             + \beta e^{\mu_2 T} \int_\Omega v_0\,dx \right)
             + \chi_1 \| w_{rr} \|_{L^\infty((0,R) \times (0,T))}.
  \]
We next estimate $J_2$.
Noticing from the third equation in \eqref{P2} 
and \eqref{wr} that
  \begin{align*}
    - w_{rr} &= \frac{n-1}{r} w_r 
                    + \frac{\alpha}{d_3} u 
                    + \frac{\beta}{d_3} v
                    - \frac{1}{d_3} \overline{M}(t)
    \\
               &\le - \frac{1}{d_3n} \overline{M}(t) 
                      + \frac{\alpha}{d_3} u 
                      + \frac{\beta}{d_3} v
    \\
               &\le \frac{\alpha}{d_3} u 
                      + \frac{\beta}{d_3} v
  \end{align*}
and using the nonpositivity of
$- d_1 \frac{n-1}{r^2}$, 
$- \frac{\chi_1}{d_3} \overline{M}(t)$,
$- \mu_1 \kappa_1 u^{\kappa_1-1}$ and  
$- a_1 \mu_1 v$, 
we infer that 
  \begin{align*}
    B_1(r,t) &\le \frac{3 \alpha \chi_1}{d_3} u + \frac{2 \beta \chi_1}{d_3} v + \mu_1
  \\
               &\le \frac{3 \alpha \chi_1}{d_3} \| u \|_{L^\infty((0,R) \times (0,T))} 
                      + \frac{2 \beta \chi_1}{d_3} \| v \|_{L^\infty((0,R) \times (0,T))} 
                      + \mu_1
               =: c_2,
  \end{align*}
and hence 
  \begin{align}\label{I_2esti}
    J_2 \le c_2 \int^R_0 r (u_r)_+^2\,dr 
  \end{align}
for all $t\in(0,T)$. 
With regard to $J_3$, 
the first condition in \eqref{mu1mu2} ensures that 
$\frac{\beta \chi_1}{d_3} - a_1 \mu_1>0$, and thereby we have
  \[
    C_1(r,t) \le \left( \frac{\beta \chi_1}{d_3} - a_1 \mu_1 \right) \| u \|_{L^\infty((0,R) \times (0,T))}.
  \]
Thanks to this inequality, $J_3$ is estimated as
  \begin{align}\label{I_3esti}
    J_3 &\le \left( \frac{\beta \chi_1}{d_3} - a_1 \mu_1 \right) 
               \| u \|_{L^\infty((0,R) \times (0,T))} 
               \int^R_0 r (u_r)_+ (v_r)_+\,dr 
         \notag \\
         &\le c_3 \int^R_0 r (u_r)_+^2\,dr
                + c_3 \int^R_0 r (v_r)_+^2\,dr
  \end{align}
for all $t\in(0,T)$, where 
$c_3:=\frac{1}{2} \big( \frac{\beta \chi_1}{d_3} - a_1 \mu_1 \big) 
\| u \|_{L^\infty((0,R) \times (0,T))}$. 
Therefore a combination of 
\eqref{I_1esti}--\eqref{I_3esti} with \eqref{I_1 I_2 I_3} 
implies that 
  \begin{align*}
    \frac{1}{2} \cdot \frac{d}{dt} \int^R_0 r (u_r)_+^2 \,dr 
    \le c_4 \int^R_0 r (u_r)_+^2\,dr + c_3 \int^R_0 r (v_r)_+^2\,dr
  \end{align*}
for all $t \in (0,T)$, where $c_4 := c_1 + c_2 + c_3$. 
Similarly, under the second condition in \eqref{mu1mu2}, 
we can observe from the second equation in \eqref{P2} that 
  \[
    \frac{1}{2} \cdot \frac{d}{dt} \int^R_0 r (v_r)_+^2 \,dr 
    \le c_5 \int^R_0 r (u_r)_+^2\,dr + c_6 \int^R_0 r (v_r)_+^2\,dr
  \]
with some $c_5=c_5(T)>0$ and $c_6=c_6(T)>0$, 
and thus find $c_7=c_7(T)>0$ such that 
  \[
    \frac{1}{2} \cdot \frac{d}{dt} \int^R_0 r ((u_r)_+^2 + (v_r)_+^2)\,dr 
    \le c_7 \int^R_0 r ((u_r)_+^2 + (v_r)_+^2)\,dr
  \]
for all $t \in (0,T)$. 
Since $u_{0r} \le 0$ and $v_{0r} \le 0$ for all $r \in [0,R]$ by \eqref{initial2}, 
the above inequality provides that
$u_r \le 0$ and $v_r \le 0$ for all $r \in (0,R)$ and $t \in (0,T)$. 
Letting $T \nearrow T_{\rm max}$, we arrive at the conclusion.
\end{proof}
%}%

We next establish a differential inequality for $\phi_U$. 
In order to derive it, we will use estimates 
as in \cite[Lemmas 4.3 and 4.4]{F_2021_optimal}. 

%================================================================%
%                                                Lemma 5.3                                               %
%================================================================%
\begin{lem}\label{lem; phi_U'}
Let $s_0 \in (0,R^n)$ and $b \in (1,2-\frac{4}{n})$. 
Assume that $\mu_1$ and $\mu_2$ satisfy \eqref{mu1mu2}. 
Then there exist $C_1>0$ and $C_2>0$ such that 
  \begin{align}\label{basic_phi_U'2}
    \phi_U'(t) &\ge - C_1 s_0^{\frac{3-b}{2}-\frac{2}{n}} \sqrt{\psi_U(t)}
  \notag \\ 
  &\quad\,
                          + \frac{\alpha \chi_1 n}{d_3} \psi_U(t)
                          + \frac{\beta \chi_1 n}{d_3} \int^{s_0}_0 s^{-b}(s_0-s) U_s(s,t)V(s,t)\,ds
  \notag \\ 
  &\quad\,
                          - C_2 s_0^{\frac{3-b}{2}} \sqrt{\psi_U(t)} 
  \notag \\ 
  &\quad\,
                          - \mu_1 n^{\kappa_1-1}
                               \int^{s_0}_0 s^{-b}(s_0-s) 
                                 \left( \int^s_0 U_s^{\kappa_1}(\sigma,t)\,d\sigma \right)\,ds
  \notag \\ 
  &\quad\,
                          - a_1 \mu_1 n 
                               \int^{s_0}_0 s^{-b}(s_0-s) 
                                 \left( \int^s_0 U_s(\sigma,t)V_s(\sigma,t)\,d\sigma \right)\,ds
  \end{align}
for all $t \in (0,\min\{1,T_{\rm max}\})$.
\end{lem}
%------------------------------proof------------------------------%
\begin{proof}
Since we can verify from the third equation in \eqref{P2} that 
  \[
    s^{1-\frac{1}{n}} w_r(s^{\frac{1}{n}},t) 
    = - \frac{\alpha}{d_3} U(s,t)
       - \frac{\beta}{d_3} V(s,t) 
       + \frac{1}{d_3 n} \overline{M}(t) s
  \]
for all $s \in (0,R^n)$ and $t \in (0,T_{\rm max})$, 
we obtain from \eqref{ut} and this identity that
  \begin{align*}
    U_t &= d_1 n^2 s^{2-\frac{2}{n}} U_{ss} 
              + \frac{\alpha \chi_1 n}{d_3} UU_s 
              + \frac{\beta \chi_1 n}{d_3} U_sV 
              - \frac{\chi_1}{d_3} \overline{M}(t) s U_s
  \\ 
  &\quad\,
              + \mu_1 U 
              - \mu_1 n^{\kappa_1-1} \int^s_0 U_s^{\kappa_1}(\sigma,t)\,d\sigma 
              - a_1 \mu_1 n \int^s_0 U_s(\sigma,t)V_s(\sigma,t)\,d\sigma
  \end{align*}
for all $s \in (0,R^n)$ and $t \in (0,T_{\rm max})$. 
In view of this identity we deduce that 
  \begin{align}\label{basic_phi_U'}
    \phi_U'(t) &\ge d_1 n^2 \int^{s_0}_0 s^{2-\frac{2}{n}-b}(s_0-s) U_{ss}\,ds
  \notag \\ 
  &\quad\,
                          + \frac{\alpha \chi_1 n}{d_3} 
                             \psi_U(t)
                          + \frac{\beta \chi_1 n}{d_3} \int^{s_0}_0 s^{-b}(s_0-s) U_sV\,ds
  \notag \\ 
  &\quad\,
                          - \frac{\chi_1}{d_3} \overline{M}(t) \int^{s_0}_0 s^{1-b}(s_0-s) U_s\,ds
  \notag \\ 
  &\quad\,
                          - \mu_1 n^{\kappa_1-1}
                               \int^{s_0}_0 s^{-b}(s_0-s) 
                                 \left( \int^s_0 U_s^{\kappa_1}(\sigma,t)\,d\sigma \right)\,ds
  \notag \\ 
  &\quad\,
                          - a_1 \mu_1 n 
                               \int^{s_0}_0 s^{-b}(s_0-s) 
                                 \left( \int^s_0 U_s(\sigma,t)V_s(\sigma,t)\,d\sigma \right)\,ds
  \notag \\ 
                  &=: \widetilde{I}_1 + \widetilde{I}_2 + \widetilde{I}_3 
                       + \widetilde{I}_4 + \widetilde{I}_5 + \widetilde{I}_6
  \end{align}
for all $t \in (0,T_{\rm max})$. 
Let us next show estimates for $\widetilde{I}_1$ and $\widetilde{I}_4$. 
With regard to $\widetilde{I}_1$, 
owing to \eqref{mu1mu2} we can employ \eqref{mean}, 
and so by an argument similar to that in the proof of 
\cite[Lemma 4.3]{F_2021_optimal} 
we provide $c_1>0$ such that
  \begin{align}\label{1inte}
    \widetilde{I}_1 \ge - c_1 s_0^{\frac{3-b}{2}-\frac{2}{n}} \sqrt{\psi_U(t)}
  \end{align}
for all $t \in (0,T_{\rm max})$.
As to $\widetilde{I}_4$, 
arguing as in \cite[Lemma 4.3]{W-2018}, 
we can take $c_2>0$ such that 
  \begin{align}\label{psiphi}
    \psi_U(t) \ge c_2 s_0^{-(3-b)} \phi_U^2(t)
  \end{align}
for all $t \in (0,T_{\rm max})$, 
so that we rely on \eqref{mean} and 
this inequality to see that 
  \begin{align}\label{step2; ineq}
    \widetilde{I}_4
    \ge - \frac{\chi_1}{d_3} \overline{M}(t) \phi_U(t)  
    \ge - \frac{\chi_1}{\sqrt{c_2} d_3} s_0^\frac{3-b}{2} \overline{M}(t) \sqrt{\psi_U(t)}
  \end{align}
for all $t \in (0,T_{\rm max})$.
Moreover, it follows from Lemma \ref{L1esti} that 
for all $t \in (0, \min\{1,T_{\rm max}\})$, 
  \[
    \overline{M}(t) \le \alpha e^{\mu_1} \int_\Omega u_0\,dx 
                              + \beta e^{\mu_2} \int_\Omega v_0\,dx,
  \]
which in conjunction with \eqref{step2; ineq} yields 
  \begin{align}\label{4inte}
    \widetilde{I}_4 \ge - c_3 s_0^{\frac{3-b}{2}} \sqrt{\psi_U(t)}
  \end{align}
for all $t \in (0, \min\{1,T_{\rm max}\})$, 
where 
$c_3:=\frac{\chi_1}{\sqrt{c_2} d_3}
         \big(\alpha e^{\mu_1} \int_\Omega u_0\,dx 
          + \beta e^{\mu_2} \int_\Omega v_0\,dx\big)$.
Accordingly, a combination of 
\eqref{basic_phi_U'}, \eqref{1inte} and \eqref{4inte} 
entails \eqref{basic_phi_U'2}.
\end{proof}
%----------------------------------------------------------------%

Next, we estimate the fifth term 
on the right-hand side of \eqref{basic_phi_U'2}. 
The following lemma generalizes 
\cite[(i) in Lemma 4.2]{F_2021_optimal}.

%================================================================%
%                                                Lemma 5.4                                               %
%================================================================%
\begin{lem}\label{UUs}
Let $s_0 \in (0,R^n)$ and $b \in (1,2)$.
Assume that $\mu_1$ and $\mu_2$ satisfy \eqref{mu1mu2}.  
Then 
  \begin{align*}
    - \mu_1 n^{\kappa_1-1} 
       \int^{s_0}_0 s^{-b}(s_0-s) 
       \left( \int^s_0 U_s^{\kappa_1}(\sigma,t)\,d\sigma \right)\,ds
    \ge
      \begin{cases} 
        - C s_0^{(3-b)\frac{2-\kappa_1}{2}} 
           \psi_U^{\frac{\kappa_1}{2}}(t)
        &\ \mbox{if} \ \kappa_1 \in (1,2),
      \\[2mm]
        - \dfrac{\mu_1 n}{b-1} 
           \psi_U(t) 
        &\ \mbox{if} \ \kappa_1=2
      \end{cases}
  \end{align*}
  for all $t \in (0,T_{\rm max})$ with some $C>0$.
\end{lem}
%------------------------------proof------------------------------%
\begin{proof} 
In the case $\kappa_1=2$ this lemma can be obtained 
by the same argument as in the proof of 
\cite[(i) in Lemma 4.2]{F_2021_optimal}. 
Therefore we consider the case $\kappa_1 \in (1,2)$. 
Noting that $b>1$, by Fubini's theorem we have 
  \begin{align}\label{lem4.3; ineq1}
    &- \mu_1 n^{\kappa_1-1} 
         \int^{s_0}_0 s^{-b}(s_0-s) 
         \left( \int^s_0 U_s^{\kappa_1}(\sigma,t)\,d\sigma \right)\,ds
    \notag \\ 
    &\quad\,
    = - \mu_1 n^{\kappa_1-1} 
          \int^{s_0}_0 \left( \int^{s_0}_\sigma s^{-b}(s_0-s)\,ds \right)
          U_s^{\kappa_1}(\sigma,t)\,d\sigma
    \notag \\
    &\quad\,
    \ge - \mu_1 n^{\kappa_1-1}
             \int^{s_0}_0 \left( \int^{s_0}_\sigma s^{-b}\,ds \right) 
             (s_0-\sigma) U_s^{\kappa_1}(\sigma,t)\,d\sigma
    \notag \\
    &\quad\,
    \ge - \frac{\mu_1 n^{\kappa_1-1}}{b-1}
             \int^{s_0}_0 s^{1-b}(s_0-s) U_s^{\kappa_1}(s,t)\,ds
  \end{align}
for all $t\in(0,T_{\rm max})$. 
Since we can apply Lemma \ref{increase} 
by the condition \eqref{mu1mu2}, 
we know that 
$U_s(s,t) \le s^{-1}U(s,t)$
for all $s \in (0,R^n)$ and $t \in (0,T_{\rm max})$, 
and hence it follows that 
  \begin{align}\label{lem4.3; ineq2}
    - \frac{\mu_1 n^{\kappa_1-1}}{b-1}
       \int^{s_0}_0 s^{1-b}(s_0-s) U_s^{\kappa_1}\,ds
    \ge - \frac{\mu_1 n^{\kappa_1-1}}{b-1}
             \int^{s_0}_0 s^{1-\frac{\kappa_1}{2}-b}(s_0-s) ( UU_s )^{\frac{\kappa_1}{2}}\,ds.
  \end{align}
Moreover, by using H\"{o}lder's inequality we can deduce that 
  \begin{align}\label{lem4.3; ineq3}
    &\int^{s_0}_0 s^{1-\frac{\kappa_1}{2}-b}(s_0-s) ( UU_s )^{\frac{\kappa_1}{2}}\,ds
    \notag \\ 
    &\quad\, 
    \le \left( \int^{s_0}_0 s^{1-b}(s_0-s)\,ds \right)^{1-\frac{\kappa_1}{2}}
         \left( \int^{s_0}_0 s^{-b}(s_0-s)UU_s\,ds \right)^\frac{\kappa_1}{2}
    \notag \\
    &\quad\,
    =   c_1 
         s_0^{(3-b)\frac{2-\kappa_1}{2}} 
         \psi_U^{\frac{\kappa_1}{2}}(t)
  \end{align}
for all $t \in (0,T_{\rm max})$, 
where $c_1:=\left( \frac{1}{(2-b)(3-b)} \right)^{(3-b)\frac{2-\kappa_1}{2}}$. 
A combination of \eqref{lem4.3; ineq1}--\eqref{lem4.3; ineq3} 
leads to the conclusion of this lemma.
\end{proof} 

As a final preparation for the proof of Theorem \ref{thm; blow-up_JL}, 
we show that the sixth term on the right-hand side of \eqref{basic_phi_U'2}
is estimated by the third term.

%================================================================%
%                                                Lemma 5.4                                               %
%================================================================%
\begin{lem}\label{UsV}
Let $s_0 \in (0,R^n)$ and $b \in (1,2)$. 
Assume that $\mu_1$ and $\mu_2$ satisfy \eqref{mu1mu2}. 
Then 
  \begin{align}\label{lem4.4; ineq}
    & - a_1 \mu_1 n 
          \displaystyle \int^{s_0}_0 s^{-b}(s_0-s) 
          \left( \int^s_0 U_s(\sigma,t) V_s(\sigma,t)\,d\sigma \right)\,ds
    \notag \\ 
    &\quad\,
    \ge 
    - \dfrac{a_1 \mu_1 n}{b-1} 
       \displaystyle \int^{s_0}_0 s^{-b}(s_0-s) U_s(s,t) V(s,t)\,ds
  \end{align}
for all $t \in (0,T_{\rm max})$.
\end{lem}
%------------------------------proof------------------------------%
\begin{proof}
As in the proof of Lemma \ref{UUs}, Fubini's theorem yields 
  \begin{align*}
    &- a_1 \mu_1 n
         \int^{s_0}_0 s^{-b}(s_0-s) 
         \left( \int^s_0 U_s(\sigma,t)V_s(\sigma,t)\,d\sigma \right)\,ds
    \\ 
    &\quad\,
    \ge - \frac{a_1 \mu_1 n}{b-1}
             \int^{s_0}_0 s^{1-b}(s_0-s) U_s(s,t)V_s(s,t)\,ds
  \end{align*}
for all $t \in (0,T_{\rm max})$.
Also, from Lemma \ref{increase} we obtain that 
$V_s(s,t) \le s^{-1}V(s,t)$
for all $s \in (0,R^n)$ and $t \in (0,T_{\rm max})$, 
which together with the above inequality implies \eqref{lem4.4; ineq}.
\end{proof}

We are now in the position to prove Theorem \ref{thm; blow-up_JL}.
%================================================================%
%                                        Proof of Theorem 5.1                                        %
%================================================================%
\begin{proof}[Proof of Theorem \ref{thm; blow-up_JL}]
We first find $b \in \left( 1,2-\frac{4}{n} \right)$ 
that there exist $c_1>0$ and $c_2>0$ such that for any $s_0 \in (0,R^n)$, 
  \begin{align}\label{ode}
    \phi_U'(t) \ge c_1 s_0^{-(3-b)} \phi_U^2(t) - c_2 s_0^{3-b-\frac{4}{n}}
  \end{align}
for all $t \in (0,\min\{ 1,T_{\rm max} \})$.
To see this, we first consider the case $\kappa_1 \in (1,2)$. 
In this case, the condition for $\chi_1$ in \eqref{JLblowup}
implies
  \begin{align}\label{case1}
    2-\frac{4}{n} 
    > 1+\frac{a_1 d_3 \mu_1}{\beta \chi_1}.
  \end{align}
Thus we can choose 
$b_1 \in \big( 1+\frac{a_1 d_3 \mu_1}{\beta \chi_1}, 2-\frac{4}{n} \big)$. 
Then, applying Lemmas \ref{lem; phi_U'}--\ref{UsV}, 
we have
  \begin{align}\label{basic_phi_U'_2}
    \phi_U'(t) &\ge - c_3 s_0^{\frac{3-b_1}{2}-\frac{2}{n}} \sqrt{\psi_U(t)} 
  \notag \\ 
  &\quad\,
                      + \frac{\alpha \chi_1 n}{d_3} 
                         \psi_U(t)
  \notag \\ 
  &\quad\,
                      + \left( \frac{\beta \chi_1 n}{d_3} - \frac{a_1 \mu_1 n}{b _1 - 1} \right)
                         \int^{s_0}_0 s^{-b_1}(s_0-s) U_sV\,ds
  \notag \\ 
  &\quad\, 
                      - c_3 s_0^{\frac{3-b_1}{2}} \sqrt{\psi_U(t)} 
  \notag \\ 
  &\quad\, 
                      - c_3 s_0^{(3-b_1)\frac{2-\kappa_1}{2}} 
                         \psi_U^{\frac{\kappa_1}{2}}(t)
  \end{align}
for all $t \in (0,\min\{1,T_{\rm max}\})$ 
with some $c_3>0$. 
Noting from the condition 
$b_1 > 1+\frac{a_1 d_3 \mu_1}{\beta \chi_1}$ that 
  \begin{align}\label{UsVnonposi}
    \left( \frac{\beta \chi_1 n}{d_3} - \frac{a_1 \mu_1 n}{b _1 - 1} \right)
    \int^{s_0}_0 s^{-b_1}(s_0-s) U_sV\,ds \ge 0
  \end{align}
and using Young's inequality 
to the first, fourth and fifth terms 
on the right-hand side of \eqref{basic_phi_U'_2}, 
we obtain $c_4>0$ such that
  \[
    \phi_U'(t) \ge \frac{\alpha \chi_1 n}{2 d_3} \psi_U(t) 
                     - c_4 s_0^{3-b_1-\frac{4}{n}} 
                        \left( 1 + 2 s_0^{\frac{4}{n}} \right)
  \]
for all $t \in (0,\min\{1,T_{\rm max}\})$. 
Hence, from \eqref{psiphi} and the fact that $s_0<R^n$ it follows that 
  \[
    \phi_U'(t) \ge c_5 
                     s_0^{-(3-b_1)}\phi_U^2(t) 
                     - c_4 \left( 1 + 2 R^4 \right) 
                        s_0^{3-b_1-\frac{4}{n}} 
  \] 
for all $t \in (0,\min\{1,T_{\rm max}\})$ with some $c_5>0$, 
which means that \eqref{ode} holds with 
$c_1=c_5$ and $c_2=c_4 \left( 1 + 2 R^4 \right)$.
Next, we consider the case $\kappa_1 = 2$. 
In light of the condition for $\chi_1$ in \eqref{JLblowup} 
we can make sure that \eqref{case1} and the relation 
$2-\frac{4}{n} > 1+\frac{d_3 \mu_1}{\alpha \chi_1}$ 
hold, 
and thereby take 
$b_2 \in \big( 1+\max\big\{ \frac{a_1}{\beta}, \frac{1}{\alpha} \big\} \cdot 
\frac{d_3 \mu_1}{\chi_1}, 2-\frac{4}{n} \big)$. 
By means of Lemmas \ref{lem; phi_U'}--\ref{UsV}, 
there exists $c_6>0$ such that 
  \begin{align*}
    \phi_U'(t) &\ge - c_6 s_0^{\frac{3-b_2}{2}-\frac{2}{n}} \sqrt{\psi_U(t)} 
  \notag \\ 
  &\quad\,
                      + \left( \frac{\alpha \chi_1 n}{d_3} - \frac{\mu_1 n}{b_2-1} \right)
                         \psi_U(t)
  \notag \\ 
  &\quad\,
                      + \left( \frac{\beta \chi_1 n}{d_3} - \frac{a_1 \mu_1 n}{b _2 - 1} \right)
                         \int^{s_0}_0 s^{-b_2}(s_0-s) U_sV\,ds
  \notag \\ 
  &\quad\, 
                      - c_6 s_0^{\frac{3-b_2}{2}} \sqrt{\psi_U(t)} 
  \end{align*}
for all $t \in (0,\min\{1,T_{\rm max}\})$. 
By the condition
$b_2>1+\max\big\{ \frac{a_1}{\beta}, \frac{1}{\alpha} \big\} \cdot \frac{d_3 \mu_1}{\chi_1}$, 
the identity \eqref{UsVnonposi} with $b_1$ replaced with $b_2$ holds. 
Furthermore we deduce that 
  \[
    \frac{\alpha \chi_1 n}{d_3} - \frac{\mu_1 n}{b_2-1} > 0.
  \]
Therefore, as in the case $\kappa_1 \in (1,2)$, 
Young's inequality and \eqref{psiphi} as well as the fact that $s_0 < R^n$
provide that
  \[
    \phi_U'(t) \ge c_7 
                     s_0^{-(3-b_2)}\phi_U^2(t) 
                     - c_8 s_0^{3-b_2-\frac{4}{n}} 
  \]
for all $t \in (0,\min\{1,T_{\rm max}\})$ with some $c_7>0$ and $c_8>0$. 

We next show that $T_{\rm max}$ is finite. 
Now we pick $b \in \big( 1,2-\frac{4}{n} \big)$ such that 
\eqref{ode} holds.
Let us fix $s_0 \in (0,R^n)$ so small satisfying 
  \[
    \sqrt{\frac{c_2}{c_1}} s_0^{1-\frac{2}{n}}
          + \frac{2}{c_1} s_0
    \le
    \frac{\widetilde{M}_0}{2^{3-b} \omega_n}, 
  \]
and hence
  \[
    \frac{\widetilde{M}_0}{2^{3-b} \omega_n} s_0^{2-b} 
    \ge \sqrt{\frac{c_2}{c_1}} s_0^{3-b-\frac{2}{n}}
          + \frac{2}{c_1} s_0^{3-b}.
  \]
Then, putting $r_\star:=\big( \frac{s_0}{4} \big)^\frac{1}{n} \in (0,R)$ 
and taking the initial data $u_0$ with \eqref{initial}, \eqref{initial2} and \eqref{initial JL}, 
we obtain that $T_{\rm max}<\infty$ as in the proof of Theorem \ref{thm; blow-up_KS}, 
which means the end of the proof.
\end{proof} 
\newpage
\bibliographystyle{plain}
\bibliography{MTY_chemotaxis-competition_2022_0807.bbl}
\end{document}